\documentclass{article}
\usepackage[lang = british]{ems-jca} 

\renewcommand{\theenumi}{{\it(\roman{enumi})}}

\def\F{\mathbb{F}}
\def\N{\mathbb{N}}

\def\HWeyl{\mathcal{H}}
\def\LieHWeyl{\mathfrak{h}}
\def\subHWeyl{\mathcal{L}}
\def\coreLie{\mathfrak{g}}
\def\coreLieC{\overline{\coreLie}}

\def\freeAlg{\F\left<\genI,\genII\right>}
\def\allWords{\left<\genI,\genII\right>}
\def\idealH{\mathcal{K}}
\def\regfac{\star}
\def\anyAlg{\mathscr{A}}
\def\anyLIdeal{\mathcal{I}}
\def\twoLIdeal{\Pi}

\def\lbrak{\left[}
\def\rbrak{\right]}
\def\lpar{\left(}
\def\rpar{\right)}
\def\lreg{\left\llbracket}
\def\rreg{\right\rrbracket}
\def\lcomp{\left<}
\def\rcomp{\right>}

\def\into{\longrightarrow}
\def\setdiff{\backslash}

\def\isoH{\varphi}
\def\canmap{\Phi_\twoGen}

\def\Canmap{\Phi}
\def\HWsetmap{\upsilon}
\def\HWCanmap{\Canmap_\HWsetmap}

\def\degBA{\deg_{\genII\genI}}

\def\genI{\alpha}
\def\genII{\beta}
\def\ImgenI{A}
\def\ImgenII{\Omega}
\def\aWord{W}
\def\twoLie{\mathcal{T}}
\def\twoGen{\chi}

\DeclareMathOperator*{\Lie}{Lie}
\DeclareMathOperator*{\ad}{ad}
\DeclareMathOperator*{\id}{id}
\DeclareMathOperator*{\Span}{Span}
\DeclareMathOperator*{\reg}{reg}
\def\freeLie{\Lie\allWords}
\def\normalLie{\mathcal{S}_{\reg}}

\def\kcoreLieDer{{\textstyle\Span_\F}\{A^n\  :\  n\in\N,\  n\geq k+1\}}

\newtheorem{theorem}{Theorem}[section]
\newtheorem{corollary}[theorem]{Corollary}
\newtheorem{lemma}[theorem]{Lemma}
\newtheorem{proposition}[theorem]{Proposition}

\newtheorem{definition}[theorem]{Definition}
\newtheorem{example}[theorem]{Example}

\begin{document}

\title{Lie structure of the Heisenberg-Weyl algebra}
\titlemark{Lie structure of the Heisenberg-Weyl algebra}

\emsauthor{1}{Rafael Reno S. Cantuba}{R.~Cantuba}


\emsaffil{1}{Department of Mathematics and Statistics, De La Salle University,\\ 2401 Taft Ave., Malate, Manila, 1004 Metro Manila, Philippines \email{rafael\_cantuba@dlsu.edu.ph}}

\classification[16Z10, 81R50]{17B60}


\begin{abstract}
As an associative algebra, the Heisenberg-Weyl algebra $\HWeyl$ is generated by two elements $A$, $B$ subject to the relation $AB-BA=1$. As a Lie algebra, however, where the usual commutator serves as Lie bracket, the elements $A$ and $B$ are not able to generate the whole space $\HWeyl$. We identify a non-nilpotent but solvable Lie subalgebra $\coreLie$ of $\HWeyl$, for which, using some facts from the theory of bases for free Lie algebras, we give a presentation by generators and relations. Under this presentation, we show that, for some algebra isomorphism $\isoH:\HWeyl\into\HWeyl$, the Lie algebra $\HWeyl$ is generated by the generators of $\coreLie$, together with their images under $\isoH$, and that $\HWeyl$ is the sum of $\coreLie$, $\isoH(\coreLie)$ and $\lbrak \coreLie,\isoH(\coreLie)\rbrak$.
\end{abstract}

\maketitle

\section{Introduction} The Heisenberg-Weyl algebra, because of its ubiquity and profundity, is said to have become the hallmark of noncommutativity in quantum theory \cite{bla08}. Motivated by the creation and annihilation operators in the traditional quantum harmonic oscillator, the Heisenberg-Weyl algebra is generated by two elements $A$ and $B$ that satisfy the canonical commutation relation $AB-BA=1$, which implies that, in the usual Hilbert space formulation of quantum mechanical systems, $A$ and $B$ may be represented by unbounded Hilbert space operators. See, for instance, \cite[Example 11.4-1]{kre89}. Nonetheless, virtually all correspondence schemes for representations of physical quantities in the Hilbert space formulation of quantum theory are said to be endowed with the Heisenberg-Weyl algebra structure \cite{bla08}. 

An approximation to the commutation relation $AB-BA=1$ was first proposed by Arik and Coon \cite{ari76}. The new commutation relation is $AB-qBA=1$, where the parameter $q$ is selected from an appropriate space such that the limit process as $q\rightarrow 1$ may be carried out. This new commutation relation resulted to bounded operator representations of $A$ and $B$ \cite[p. 524]{ari76}. The new model has been successfully applied to several fields including particle physics, knot theory and general relativity \cite[Chapter 12]{ern12}. To mention one specific example, there was a study \cite{mon96} on a Helium isotope in which a model based on the commutation relation $AB-qBA=1$ was compared with the available experimental data, and the computed spectrum reproduces the experimental one within less than 5\% discrepancy \cite[p. 1100]{mon96}.

Thus, the Heisenberg-Weyl algebra now belongs to a family of algebras $\HWeyl_q$ generated by two elements $A$, $B$ subject to the relation $AB-qBA=1$. We call $\HWeyl_q$ the \emph{$q$-deformed Heisenberg algebra}. We mention here two perspectives on the combinatorial algebra of $q$-deformed Heisenberg algebras, that have appeared in the literature. 

The first is in terms of algebraic term rewriting\footnote{16S15 in the 2020 MSC} \cite{hel00,hel02,hel05}. In these studies, the focus was on the rewriting system or reduction system induced by the relation $AB-qBA=1$ that are used to arrive at the traditional ``normal form'' for a given $F\in\HWeyl$, which in this case, is when $F$ has been expressed as a linear combination of words $B^mA^n$ where $m$, $n$ are nonnegative integers. This reduction system and the corresponding normal form were used in \cite{hel00} to study centralizers of elements of $\HWeyl_q$ and the algebraic dependence of commuting elements, while in \cite{hel05}, the structure of two-sided ideals of $\HWeyl_q$ was studied using deformed commutator mappings. In \cite{hel02}, the generalization of $\HWeyl$ into $\HWeyl_q$ was an important running example on how the Diamond Lemma for Ring Theory \cite{ber78} was generalized from its usual ring-theoretic scope into classes of power series algebras. 

The second perspective on the study of $q$-deformed Heisenberg algebras concerns a nonassociative structure, or more precisely, a Lie algebra structure, induced by the operation $\HWeyl_q\times\HWeyl_q\into\HWeyl_q$ given by the the usual commutator $\lpar F,G\rpar\mapsto FG-GF$ \cite{can19,can20a,can20b,can21}. One main theorem about this is that, if $q\neq 1$, then the Lie subalgebra of $\HWeyl_q$ generated by $A$, $B$ consists of all linear combinations of $A$, $B$, $B^mA^n$, where $mn\neq 0$. The determination of such Lie subalgebra is said to be the solution to the \emph{Lie polynomial characterization problem} \cite{can20c} for $\HWeyl$ under the usual generators and relation. An algebraic solution to this was done in \cite{can19} when $q$ is not a root of unity, and in \cite{can20b} when $q\neq 1$ is a root of unity. Alternatively, if $q$ is in the real interval $(0,1)$, then an operator-theoretic solution was made in \cite{can20a}. The methods in \cite{can19} were also used in \cite{can21} for a central extension of the algebra $\HWeyl_q$.

Let us now consider some concrete examples. As mentioned earlier, the associative algebra $\HWeyl_q$ may be turned into a Lie algebra with Lie bracket $\lbrak F,G\rbrak = FG-GF$ for any $F,G\in\HWeyl$. Consider the elements
\begin{eqnarray}
X & = &\lbrak\lbrak B,A\rbrak,\lbrak\lbrak B,A\rbrak,A\rbrak\rbrak,\nonumber\\
Y & = &  \lbrak B,\lbrak\lbrak B,\lbrak B,A\rbrak\rbrak,\lbrak B,A\rbrak\rbrak\rbrak,\nonumber\\
Z & = &\lbrak B,\lbrak\lbrak B,A\rbrak,\lbrak\lbrak B,A\rbrak, A\rbrak\rbrak\rbrak,\nonumber
\end{eqnarray}
of $\HWeyl_q$. If $q$ is not a root of unity, then using results from \cite{can19}, $X$ is a linear combination of $B^2A^3$, $BA^2$ and $1$; $Y$ is a linear combination of $B^3A^2$, $B^2A$ and $1$, while $Z$ is a linear combination of $B^3A^3$, $B^2A^2$, $BA$ and $1$. However, it is not possible to express elements like $B^3$, $A^3$, $B^2$, $A^2$ (pure powers of $A$ or $B$ with exponent at least 2) in terms of only Lie algebra operations performed on the generators $A$, $B$. Properties, such as these, of some Lie algebra structure or \emph{Lie structure} on $\HWeyl_q$ has led to some interesting results, one of which is the characterization of compact elements of $\HWeyl_q$ (under some operator norm) leading to a Calkin algebra isomorphic to an algebra of Laurent polynomials in one variable \cite{can20a}. 

However, when we take the limit as $q\rightarrow 1$, in the Heisenberg-Weyl algebra $\HWeyl=\HWeyl_1$, the aforementioned Lie structure reduces the linear span of $1$, $A$ and $B$. But still, there is more to the Lie algebra $\HWeyl$ than just the elements $c_1\cdot 1+c_2 A+c_3B$ for all scalars $c_1,c_2,c_2$. \emph{This is the starting point of our inquiry.}  If the Lie structure of the Heisenberg-Weyl algebra cannot be   studied by solving a Lie polynomial characterization problem (because the solution is almost trivial), then how can the rest of the Lie algebra $\HWeyl$ (outside the span of $1$, $A$, $B$) be decribed? 

In this work, we answer this question by expressing the Lie algebra $\HWeyl$ as the sum of three vector subspaces. The first summand is some Lie subalgebra $\coreLie$ (to be defined in Section~\ref{ConcreteSec}), with the second summand being the image of $\coreLie$ under some algebra isomorphism $\isoH$, and the third summand is $\lbrak \coreLie,\isoH(\coreLie)\rbrak$. If the Lie subalgebra $\coreLie$ is of such importance in elucidating the Lie structure of $\HWeyl$, then we naturally want to know more about it. What we did in this work is to give a presentation for $\coreLie$ by generators and relations. To do this, we first give, in Section~\ref{freeSec}, a treatment of selected aspects of the theory of Lyndon-Shirshov words, and of the role they play in the basis theory for free Lie algebras.

For the sake of completeness, we mention here some studies on the Lie structure of some classes of associative algebras \cite{mon17,ril93}, and of a certain special product of associative algebras \cite{sic17}. These studies are focused on necessary and sufficient conditions for nilpotency or solvability of the desired Lie algebras in a field with nonzero characteristic. These studies also involve results that are in the framework of polynomial identity algebras. Although we shall be showing the non-nilpotence and solvability of the Lie algebra $\coreLie$, which was mentioned earlier as our key in describing the Lie structure of $\HWeyl$, the subject of this work differs significantly from the said approach in \cite{mon17,ril93,sic17}. In a later result, we will specifically require the underlying field to have zero characteristic; otherwise, some trivialities will be introduced in the Lie structure of $\coreLie$, and hence of $\HWeyl$. Also, instead of focusing on polynomial identities, we delve deeper into the combinatorial algebra of Lyndon-Shirshov words, and the properties of the free Lie algebra basis that can be derived from them.

\section{Preliminaries} If the two-element set is denoted by $\{\genI,\genII\}$, then we shall refer to $\genI$ and $\genII$ as \emph{words of length $1$}. If, for some positive integer $n$, all words of length strictly less than $n$ have been defined, then by a \emph{word of length $n$}, we mean any juxtaposition of the form $\aWord_1\aWord_2$, where one of $\aWord_1$ or $\aWord_2$ is a word of length $n-1$, and the other is a word of length $1$. If $\aWord$ is a word of length $n$, but mention of the positive integer $n$ is not relevant in the current context, then we simply refer to $\aWord$ as a \emph{word}, or a \emph{word on $\{\genI,\genII\}$}, or a \emph{word on $\genI$, $\genII$}. Conversely, if $\aWord$ is referred to as a word, then it is assumed that $\aWord$ is a word of length $n$ for some positive integer $n$. In such a case, we define $n$ as the \emph{length} of $\aWord$, which we denote by $|\aWord|$. Any word $\aWord$ may be written as $\aWord=X_1X_2\cdots X_n$, where, for each $k\in\{1,2,\ldots,n\}$, we have $X_k\in\{\genI,\genII\}$. Induction on $n$ may be used to prove that $|\aWord|=n$. Suppose $\aWord'=Y_1Y_2\cdots Y_m$ is also a word with $Y_k\in\{\genI,\genII\}$ for any $k\in\{1,2,\ldots,m\}$. Again, by induction, $m=|\aWord'|$. \emph{Equality of words}, or in this case, $\aWord=\aWord'$, is defined by the conditions $|\aWord|=|\aWord'|$, and $X_k=Y_k$ for any $k\in\{1,2,\ldots,n\}$. By the \emph{support} of $\aWord$, we mean the set $\mathcal{X}\subseteq\{\genI,\genII\}$ such that for each $k\in\{1,2,\ldots,n\}$, $X_k\in\mathcal{X}$. We define the \emph{empty word} or \emph{word of length $0$} as the word with empty support. We use the symbol $1$ to denote the empty word, and we define $|1|:=0$. A word $W$ is \emph{nonempty} if $W\neq 1$. Given a positive integer $t$, the word $\aWord^t$ is the juxtaposition of $\aWord$ with itself so that $\aWord$ appears in the juxtaposition $t$ times, just like the exponentiation in elementary number systems. We interpret $\aWord^0$ as the empty word. For each $n\in\N:=\{0,1,2\ldots\}$, let $\allWords_n$ be the set of all words of length $n$. The set of all words on $\{\genI,\genII\}$, which is $\allWords:=\bigcup_{n\in\N}\allWords_n$, may easily be shown to be a noncommutative monoid under the operation of juxtaposition of words, with the empty word as (multiplicative) identity. 

Let $\F$ be a field. We assume that any $\F$-algebra to be mentioned is unital and associative. Since we shall not be considering any set of scalars other than $\F$, we further drop the prefix ``$\F$-'' and so we shall simply use the term ``algebra.''  Let $\freeAlg$ be the free algebra generated by the two-element set $\{\genI,\genII\}$. The words on $\{\genI,\genII\}$ form a basis for $\freeAlg$, as a vector space over $\F$. We assume that no confusion shall arise in using the same symbol for the empty word and the multiplicative identity of the field $\F$.  

The \emph{Heisenberg-Weyl algebra} is the algebra $\HWeyl$ generated by two elements $A$, $B$ satisfying the relation $AB=BA+1$. By the universal property of the free algebra $\freeAlg$, $\HWeyl$ is isomorphic to some quotient of $\freeAlg$. More precisely, if $\idealH$ is the (two-sided) ideal of $\freeAlg$ generated by $-\genI\genII+\genII\genI+1$, then $\HWeyl$ is isomorphic to $\freeAlg/\idealH$. 

Throughout, if a vector space basis is known for an algebra or Lie algebra $\mathcal{A}$, then this basis is understood to be a Hamel basis. That is, regardless of whether $\mathcal{A}$ is infinite dimensional or not, $\mathcal{A}$ is viewed as the set of all finite linear combinations of said basis elements. Some facts about a traditional basis for the Heisenberg-Weyl algebra are discussed at the beginning of Section~\ref{ConcreteSec}. 

\subsection{Nested adjoint maps}\label{admapSec} Any algebra $\anyAlg$ is a Lie algebra under the operation $\anyAlg\times\anyAlg\into\anyAlg$ given by\linebreak $(X,Y)\mapsto \lbrak X,Y\rbrak:=XY-YX$. Given $X\in\anyAlg$, the \emph{adjoint map} $\ad X$ is the linear map $\anyAlg\into\anyAlg$ given by $Y\mapsto \lbrak X,Y\rbrak$. 
The adjoint map gives a convenient notation for nested Lie brackets that is precise, and with no need for vague use of expressions like ``$n$ times,'' or ``$n$ copies.'' For instance, 
\begin{eqnarray}
 \lbrak X,\lbrak X,\lbrak X,\lbrak X,\lbrak X, Y\rbrak \rbrak \rbrak\rbrak\rbrak & = & (\ad X)^5(Y),\nonumber\\
\lbrak\lbrak\lbrak X,Y\rbrak , Y\rbrak , Y\rbrak & = & (-\ad Y)^3(X),\nonumber\\
 \lbrak X,\lbrak X,\lbrak X,\lbrak X,\lbrak X, \lbrak\lbrak\lbrak X,Y\rbrak , Y\rbrak , Y\rbrak \rbrak \rbrak \rbrak\rbrak\rbrak & = & (\ad X)^5(-\ad Y)^3(X),\nonumber\\
\lbrak\lbrak\lbrak  \lbrak X,\lbrak X,\lbrak X,\lbrak X,\lbrak X, Y\rbrak \rbrak \rbrak\rbrak\rbrak,Y\rbrak , Y\rbrak , Y\rbrak & = & (-\ad Y)^3(\ad X)^5(Y),\nonumber
\end{eqnarray}
where juxtaposition and exponentiation of adjoint maps refer to function composition. Given $m,n\in\N$, the reader may infer the meaning of generalized nested Lie brackets like $(\ad X)^m(Y)$, $(-\ad Y)^m(X)$, $(\ad X)^m(-\ad Y)^n(X)$ or $(-\ad Y)^m(\ad X)^n(Y)$, and perhaps these examples may show the advantage of the ``adjoint notation'' for nested Lie brackets.

\section{Nonassociative regular words on two generators}\label{freeSec}

In this section, we give a rigorous treatment of the aspects of the theory of regular words on two generators. These objects were motivated by notions from, and have their crucial significance in, several algebraic theories, mainly the theory of presentation of groups, the so-called ``Fox calculus'' or the free differential calculus, and also the theory of bases for free Lie algebras \cite{che58,shi09a,shi09b}. Some excellent modern expositions are \cite{bok07} and \cite[Sections 2.2, 2.7--2.9]{ufn95}. Our treatment here is mainly based on \cite{ufn95} because of the agreeable perspective in it: the said theoretical developments can be dealt with in principle on the associative level, with the aid of universal enveloping algebras, but it makes sense, however, to do so, not outside the scope of the Lie algebras themselves \cite[p. 37]{ufn95}.

Let $V,W\in\allWords$. We say that $V$ is a \emph{subword} of $W$ if there exist $L,R\in\allWords$ such that $W=LVR$. If $L=1$, then $V$ is a \emph{beginning} of $W$, and is an \emph{ending} of $W$, if $R=1$. A subword $U$ of $W$ is \emph{proper} if $U\neq W$. Suppose $\{U\  :\  \mathscr{P}(U)\}\subseteq\allWords$ for some statement $\mathscr{P}$. A \emph{longest} word with property $\mathscr{P}$ is an element $U'$ of $\{U\  :\  \mathscr{P}(U)\}$ such that for any $U\in \{U\  :\  \mathscr{P}(U)\}$, $|U|\leq |U'|$. Any nonempty word has a unique longest ending, and any word with length at least $2$ has a unique longest proper ending.

Let $n=|VW|=|WV|$, and let $VW=X_1X_2\cdots X_n$, $WV=Y_1Y_2\cdots Y_n$, with\linebreak $X_i,Y_i\in\{\genI,\genII\}$ for any $i\in\{1,2,\ldots,n\}$. If we define $>$ as the ordering on $\{\genI,\genII\}$ given by $\genI>\genII$, then we may extend $>$ to an ordering in $\allWords$ by defining $V>W$ whenever there exists $k\in\{1,2,\ldots,n\}$ such that $X_k>Y_k$ and $i<k$ implies $X_i=Y_i$. If $V>W$ or $V=W$, then we write $V\geq W$.

\begin{definition} A nonempty word $W\in\allWords$ is \emph{regular} if for any proper beginning $L$ and any proper ending $R$ of $W$ such that $W=LR$, we have $L>R$.
\end{definition}

The notion of a regular word is one of the fundamental cornerstones of the basis theory of free Lie algebras. We shall gradually introduce properties of regular words according to what shall be relevant to the Lie structure of the Heisenberg-Weyl algebra. First, we have the following property which is useful in constructing some regular words longer than the generators $\genI$ and $\genII$. 

\begin{proposition}[{\cite[Theorem 2.8.1]{ufn95}}]\label{regmultProp} If $W_1$ and $W_2$ are regular words such that \linebreak $W_1>W_2$, then $W_1W_2$ is regular.
\end{proposition}

We shall be concerned with specific types of regular words that are summarized in the following.

\begin{example}\label{Illus1Ex} Proposition~\ref{regmultProp} has the following consequences, the proofs of which are routine.
\begin{enumerate}\item\label{mnform} By induction, the word $\genI^m\genII^n$ is regular for any $m,n\in\N\setdiff\{0\}$ .
\item\label{mnmnform} Given positive integers $h,m,n$, the word $\genI^h\genII^m\genI^h\genII^n$ is regular if and only if $m<n$.
\item\label{hkform} Given positive integers $h,k,m,n$, the word $\genI^{h+k}\genII^m\genI^h\genII^n$ is regular, but $\genI^{h}\genII^m\genI^{h+k}\genII^n$ is not.
\item\label{genregword} Given positive integers $m_1,n_1,m_2,n_2,\ldots,m_k,n_k$, if, for any index\linebreak $i\in\{2,3,\ldots,k\}$, we have $m_1> m_i$, then $\genI^{m_1}\genII^{n_1}\genI^{m_2}\genII^{n_2}\cdots \genI^{m_k}\genII^{n_k}$ is regular.
\item\label{genregword2} Let $m_1,n_1,m_2,n_2,\ldots,m_k,n_k$ be positive integers such that\linebreak $m_1=\max\{m_i\  :\  i\in\{2,3,\ldots,k\}\}$. Let $s\in\{2,3,\ldots,k\}$ such that
\begin{eqnarray}
m_s & = &\max\{m_i\  :\  i\in\{2,3,\ldots,k\}\},\nonumber\\
m_s & > & m_i,\qquad\mbox{if } i\in\{2,3,\ldots, s-1\}.\nonumber
\end{eqnarray}
That is, the biggest value among $m_2,m_3,\ldots,m_k$ has its first occurrence at index $s$. If $n_1<n_s$, then $\genI^{m_1}\genII^{n_1}\genI^{m_2}\genII^{n_2}\cdots \genI^{m_k}\genII^{n_k}$ is regular.  If $n_1>n_s$, then $\genI^{m_1}\genII^{n_1}\genI^{m_2}\genII^{n_2}\cdots \genI^{m_k}\genII^{n_k}$ is NOT regular.
\end{enumerate}
\end{example}

\begin{example}\label{powernotregEx} Given an integer $n\geq 2$, if $L$ is a proper beginning of the word $\genII^n$, then $L=\genII^i$ for some positive integer $i<n$. The proper ending $R$ of $\genII^n$ such that $\genII^n=LR$ is $R=\genII^{n-i}$. Thus, $LR = \genII^n = RL$, and so, $L{\not >}R$. This means that $\genII^n$ is not regular.
\end{example}

\subsection{Regular factoring} Perhaps the most important property of regular words is the existence of one unique way to ``factor'' a regular word such that the subwords in the ``factoring'' are also regular, and the ``factorization'' process may be continued on these factors repeatedly, and the process terminates until all factors are generators. This feature of regular words makes them of extreme significance to a particular nonassociative structure in $\freeAlg$ that we will be using later.

\begin{lemma}[{\cite[Theorem 2.8.3(b)]{ufn95}}]\label{regLem} If $W$ is a regular word, if $R$ is the longest regular proper ending of $W$, and if $L$ is the proper beginning of $W$ such that $W=LR$, then $L$ is regular. The uniqueness of $R$ implies the uniqueness of the pair $(L,R)$, which we henceforth refer to as the \emph{regular factoring} of $W$. In symbols, we write this as $W=L\regfac R$.
\end{lemma}

\begin{example}\label{Illus2Ex} The regular words mentioned in Example~\ref{Illus1Ex} have the following regular factorings.
\begin{enumerate}\item\label{firstform} For the case $m=1$ in Example~\ref{Illus1Ex}\ref{mnform}, for any positive integer $n$, we have $\genI\genII^n=\genI\genII^{n-1}\regfac\genII$.
\item\label{firstform2} For any positive integers $m,n$, $\genI^m\genII^n=\genI\regfac\genI^{m-1}\genII^n$.
\item\label{secondform} Given positive integers $h,m,n$, with $m<n$, $\genI^h\genII^m\genI^h\genII^n=\genI^h\genII^m\regfac\genI^h\genII^n$.
\item\label{secondform2} Given positive integers $h,k,m,n$, $\genI^{h+k}\genII^m\genI^h\genII^n=\genI\  \regfac\  \genI^{h+k-1}\genII^m\genI^h\genII^n$.
\item\label{genregword3} Given positive integers $m_1,n_1,m_2,n_2,\ldots,m_k,n_k$, if, for any index\linebreak $i\in\{2,3,\ldots,k\}$, we have $m_1-1> m_i$, then 
\begin{eqnarray}
\genI^{m_1}\genII^{n_1}\genI^{m_2}\genII^{n_2}\cdots \genI^{m_k}\genII^{n_k}=\genI\regfac\genI^{m_1-1}\genII^{n_1}\genI^{m_2}\genII^{n_2}\cdots \genI^{m_k}\genII^{n_k}.\label{genregwordEQ3}
\end{eqnarray}
\end{enumerate}
\end{example}

Since any nonempty proper ending of $\genI\genII^n$ is $\beta^i$ for some positive integer $i$, as shown in Example~\ref{powernotregEx}, none of these nonempty proper endings is regular except for $\genII$ itself. This explains the regular factoring $\genI\genII^n=\genI\genII^{n-1}\regfac\genII$ in Example~\ref{Illus2Ex}\ref{firstform}. The longest proper ending of the regular word $\genI^2\genII^n$ is $\genI\genII^n$, which is already regular. Thus, $\genI^2\genII^n=\genI\regfac\genI\genII^n$, and this may be extended by induction. The result is Example~\ref{Illus2Ex}\ref{firstform2}. As for Example~\ref{Illus2Ex}\ref{secondform}, a routine argument may be used to show that any proper ending of $\genI^h\genII^m\genI^h\genII^n$ longer than $\genI^h\genII^n$ is not regular. Example~\ref{Illus2Ex}\ref{secondform} may be extended to Example~\ref{Illus2Ex}\ref{secondform2}--\ref{genregword3} in the same manner as how Example~\ref{Illus2Ex}\ref{firstform} was extended to Example~\ref{Illus2Ex}\ref{firstform2}. 

\subsection{Regular bracketing} The free algebra $\freeAlg$ is a Lie algebra under the Lie bracket 
\begin{eqnarray}
(f,g)\mapsto \lbrak f,g\rbrak:=fg-gf,\nonumber
\end{eqnarray}
and we have the following construction of elements of the Lie algebra $\freeAlg$ where the ``nesting of Lie brackets'' is elegantly encoded in the structure of regular words.

\begin{definition} Define $\lreg\genI\rreg := \genI$ and $\lreg\genII\rreg := \genII$. If $W$ is a regular word with length at least $2$ and if, using Lemma~\ref{regLem}, $W=L\regfac R$, then $\lreg W\rreg:=\lbrak \lreg L\rreg,\lreg R\rreg\rbrak$. We shall refer to $\lreg W\rreg$ as the \emph{regular bracketing} of $W$. We say that $f\in\freeAlg$ is a \emph{nonassociative regular word (on $\genI$, $\genII$)} if there exists a regular word $W$ such that $f=\lreg W\rreg$.
\end{definition}

\begin{example}\label{Illus3Ex}\begin{enumerate}\item\label{firstformBrak} Using Example~\ref{Illus2Ex}\ref{firstform}--\ref{firstform2} and induction, for any $m,n\in\N$ with $m\geq 1$, we have 
\begin{eqnarray}
\lreg\genI^m\genII^n\rreg=(\ad\genI)^{m-1}(-\ad\genII)^n(\genI).\nonumber
\end{eqnarray}
[See Section~\ref{admapSec} for explanatory remarks on the use of nested adjoint maps.]
\item\label{secondformBrak} As a consequence of \ref{firstformBrak} above, and also of Example~\ref{Illus2Ex}\ref{secondform}, given positive integers $h,m,n$ with $m<n$,
\begin{eqnarray}
\lreg\genI^h\genII^m\genI^h\genII^n\rreg=\lbrak(\ad\genI)^{h-1}(-\ad\genII)^m(\genI),(\ad\genI)^{h-1}(-\ad\genII)^n(\genI)\rbrak.\nonumber
\end{eqnarray}
\item\label{secondformBrak2} Using Example~\ref{Illus2Ex}\ref{secondform2}, and induction, given positive integers $h,k,m,n$ with $m<n$,
{\scriptsize\begin{eqnarray}
\lreg\genI^{h+k}\genII^m\genI^h\genII^n\rreg=(\ad\genI)^{k}\lpar\lbrak(\ad\genI)^{h-1}(-\ad\genII)^m(\genI),(\ad\genI)^{h-1}(-\ad\genII)^n(\genI)\rbrak\rpar.\nonumber
\end{eqnarray}}
\item\label{secondformBrak3} Using Example~\ref{Illus2Ex}\ref{secondform2}, and induction,  given positive integers $h,k,m,n$ with $m\geq n$,
{\scriptsize\begin{eqnarray}
\lreg\genI^{h+k}\genII^m\genI^h\genII^n\rreg=(\ad\genI)^{k-1}\lpar\lbrak(\ad\genI)^{h}(-\ad\genII)^m(\genI),(\ad\genI)^{h-1}(-\ad\genII)^n(\genI)\rbrak\rpar.\nonumber
\end{eqnarray}}
\item\label{genformBrak} We may generalize or combine \ref{secondformBrak}--\ref{secondformBrak3} above, and make some substitutions using \ref{firstformBrak}, to obtain
\begin{eqnarray}
\lreg\genI^{h+k}\genII^m\genI^h\genII^n\rreg=\begin{cases}\lpar\ad\genI\rpar^k\lbrak\lreg\genI^h\genII^m\rreg,\lreg\genI^h\genII^n\rreg\rbrak, &\mbox{if } m<n,\\
\lpar\ad\genI\rpar^{k-1}\lbrak\lreg\genI^{h+1}\genII^m\rreg,\lreg\genI^h\genII^n\rreg\rbrak, &\mbox{if }k\geq 1,\  m\geq n.\end{cases}\nonumber
\end{eqnarray}
The case $k=0$ with $m\geq n$ is not included because in such a case, according to Example~\ref{Illus1Ex}\ref{mnmnform}, $\genI^{h+k}\genII^m\genI^h\genII^n$ is not regular.
\end{enumerate}
\end{example}

We give a few remarks on how Example~\ref{Illus3Ex}\ref{secondformBrak} has been generalized into\linebreak  Example~\ref{Illus3Ex}\ref{secondformBrak2}--\ref{secondformBrak3}. Consider the word $\genI^{h+1}\genII^m\genI^h\genII^{n}$, where $h$, $m$, $n$ are positive. Proposition~\ref{regmultProp} and Example~\ref{Illus1Ex}\ref{firstform} may be used to show that $\genI^{h+1}\genII^m\genI^h\genII^{n}$ is regular. This is regardless of which of $m$ or $n$ is bigger. If $m<n$, then the longest proper ending of $\genI^{h+1}\genII^m\genI^h\genII^{n}$ is $\genI^{h}\genII^m\genI^h\genII^{n}$, which, by Example~\ref{Illus1Ex}\ref{mnmnform} is already regular. Thus, 
\begin{eqnarray}
\lreg\genI^{h+1}\genII^m\genI^h\genII^{n}\rreg & = & \lreg\genI\regfac\genI^{h}\genII^m\genI^h\genII^{n}\rreg,\nonumber\\
& = & \lbrak\lreg\genI\rreg,\lreg\genI^{h}\genII^m\genI^h\genII^{n}\rreg\rbrak,\nonumber\\
& = & \lbrak\genI,\lbrak(\ad\genI)^{h-1}(-\ad\genII)^m(\genI),(\ad\genI)^{h-1}(-\ad\genII)^n(\genI)\rbrak\rbrak,\nonumber\\
& = &(\ad\genI)\lpar\lbrak(\ad\genI)^{h-1}(-\ad\genII)^m(\genI),(\ad\genI)^{h-1}(-\ad\genII)^n(\genI)\rbrak\rpar,\nonumber
\end{eqnarray}
which is the formula in Example~\ref{Illus3Ex}\ref{secondformBrak2} at $k=1$. This may be extended to an arbitrary positive integer $k$ by induction. For the case $m\geq n$, a routine argument may be used to show that any proper ending of $\genI^{h+1}\genII^m\genI^h\genII^{n}$ longer than $\genI^h\genII^{n}$ is not regular. Thus,
\begin{eqnarray}
\lreg\genI^{h+1}\genII^m\genI^h\genII^{n}\rreg & = & \lreg\genI^{h+1}\genII^m\regfac\genI^h\genII^{n}\rreg,\nonumber\\
& = &\lbrak\lreg\genI^{h+1}\genII^m\rreg,\lreg\genI^h\genII^{n}\rreg\rbrak,\nonumber\\
& = &\lbrak(\ad\genI)^{h}(-\ad\genII)^m(\genI),(\ad\genI)^{h-1}(-\ad\genII)^n(\genI)\rbrak,\nonumber
\end{eqnarray}
which is the formula in Example~\ref{Illus3Ex}\ref{secondformBrak3} at $k=1$. At the next value of $k$, the longest proper ending of the word $\genI^{h+2}\genII^m\genI^h\genII^{n}$ is $\genI^{h+1}\genII^m\genI^h\genII^{n}$, which we have established to be regular. This shall result to the formula in Example~\ref{Illus3Ex}\ref{secondformBrak3} at $k=2$. Using Example~\ref{Illus2Ex}\ref{secondform2} and induction, this may be extended to any positive $k\in\N$. 

\subsection{The general regular word on two generators}\label{genregwordSec} After some facts about regular factoring and regular bracketing in the previous subsections, we now consider how these notions may be understood for a regular word (on two generators) of arbitrary length. To this end, we have an important necessary condition for the regularity of a word in the lemma that follows. Also, this gives us a definite form of a regular word on two generators, and we use this form to partition the collection of regular words, which shall be relevant in our main results later.

\begin{lemma}\label{badegLem} If $W\in\allWords$ is nonempty and regular with length at least $2$, then there exist positive integers $m_1,n_1,m_2,n_2,\ldots,m_k,n_k$ such that 
\begin{eqnarray}
W=\genI^{m_1}\genII^{n_1}\genI^{m_2}\genII^{n_2}\cdots \genI^{m_k}\genII^{n_k}.\label{necessW}
\end{eqnarray}
\end{lemma}
\begin{proof} First, we claim that there exist distinct $\gamma,\delta\in\{\genI,\genII\}$, such that\linebreak $W=\gamma^{m_1}\delta^{n_1}\gamma^{m_2}\delta^{n_2}\cdots \gamma^{m_k}\delta^{n_k}$, and we prove this by induction on $|W|\geq 2$. If $|W|=2$, then the only possible words equal to $W$ are $\genI^2$, $\genI\genII$, $\genII^2$ and $\genII\genI$. The only proper beginning of $\alpha\beta$ is $L=\alpha$ and the proper ending $R$ such that $W=LR$ is $R=\beta$, and we also have $L>R$. Thus, $\alpha\beta$ is regular. For each of the other three cases for $W$, there exists a proper beginning $L'$ and proper ending $R'$ such that $W=L'R'$ but $L'{\not >}R'$. Thus, the only possibility is $W=\alpha\beta$, which satisfies the statement. Suppose that the statement holds for any regular word with length strictly less than $|W|$. If $W=L\regfac R$, then by the inductive hypothesis, there exist distinct $\gamma,\delta\in\{\genI,\genII\}$ such that $L=\gamma^{m_1}\delta^{n_1}\gamma^{m_2}\delta^{n_2}\cdots \gamma^{m_k}\delta^{n_k}$, and also some distinct $\varepsilon,\zeta\in\{\genI,\genII\}$ such that $R=\varepsilon^{m_{k+1}}\zeta^{n_{k+1}}\varepsilon^{m_{k+2}}\zeta^{n_{k+2}}\cdots \varepsilon^{m_{k+\ell}}\zeta^{n_{k+\ell}}$ where all exponents shown are positive. We have either $\varepsilon=\gamma$ or $\varepsilon=\delta$, and these cases imply $\zeta=\delta$ or $\zeta=\gamma$, respectively. Thus,
\begin{eqnarray}
W & = & \gamma^{m_1}\delta^{n_1}\gamma^{m_2}\delta^{n_2}\cdots \gamma^{m_k}\delta^{n_k}\gamma^{m_{k+1}}\delta^{n_{k+1}}\gamma^{m_{k+2}}\delta^{n_{k+2}}\cdots \gamma^{m_{k+\ell}}\delta^{n_{k+\ell}},\  \mbox{or}\nonumber\\
W & = & \gamma^{m_1}\delta^{n_1}\gamma^{m_2}\delta^{n_2}\cdots \gamma^{m_k}\delta^{n_k+m_{k+1}}\gamma^{n_{k+1}}\delta^{m_{k+2}}\gamma^{n_{k+2}}\cdots \delta^{m_{k+\ell}}\gamma^{n_{k+\ell}},\nonumber
\end{eqnarray}
where all the exponents shown are positive. Thus, in any case, $W$ satisfies the desired conditions. This completes the induction proof for the claim. What remains to be shown is $\gamma=\genI$. Suppose otherwise. Then the only possibility is $\gamma=\genII$, and so,\linebreak $W=\genII^{m_1}\genI^{n_1}\genII^{m_2}\genI^{n_2}\cdots \genII^{m_k}\genI^{n_k}$. Consequently, the proper beginning $L=\genII^m_1$ and proper ending $R=\genI^{n_1}\genII^{m_2}\genI^{n_2}\cdots \genII^{m_k}\genI^{n_k}$ of $W$ have the property that $W=LR$ but since $m_1$ and $n_1$ are positive, we may write
\begin{eqnarray}
LR & = & \genII\cdot\genII^{m_1-1}\genI^{n_1}\genII^{m_2}\genI^{n_2}\cdots \genII^{m_k}\genI^{n_k},\nonumber\\
RL & = & \genI\cdot\genI^{n_1-1}\genII^{m_2}\genI^{n_2}\cdots \genII^{m_k}\genI^{n_k}\genII^{m_1},\nonumber
\end{eqnarray}
and paying attention to the generator at the left-most positions in these words, $\genII{\not >}\genI$, and so $LR{\not>}RL$. This contradicts the regularity of $W$. Therefore, $\gamma=\genI$, and\linebreak $W=\genI^{m_1}\genII^{n_1}\genI^{m_2}\genII^{n_2}\cdots \genI^{m_k}\genII^{n_k}$.
\end{proof}

Given words $W$ and $V$, the number of times $V$ occurs as a subword of $W$ is denoted by $\deg_VW$. With reference to the notation in Lemma~\ref{badegLem}, because the exponents $m_1,n_1,m_2,n_2,\ldots,m_k,n_k$ are all positive, the arbitrary regular word $W$ may be written as
{\scriptsize\begin{eqnarray}
W=\genI^{m_1}\genII^{n_1-1}\cdot\  \genII\genI\  \cdot\genI^{m_2-1}\genII^{n_2-1}\cdot\  \genII\genI\  \cdots\  \genI^{m_{k-1}-1}\genII^{n_{k-1}-1}\cdot\  \genII\genI\cdot\  \genI^{m_k-1}\genII^{n_k},\label{degBAdef0}
\end{eqnarray}}

\noindent and so
\begin{eqnarray}\deg_{\genII\genI}W=k-1.\label{degBAdef}
\end{eqnarray}

We now give some remarks concerning the regular factoring and regular bracketing of the regular word \eqref{necessW}. Suppose that the biggest value among $m_2,m_3,\ldots,m_k$ has its first occurrence at index $s$. By a routine argument, the regularity of $W$ implies
\begin{eqnarray}
m_1\geq m_s,\nonumber
\end{eqnarray}
and so, if
\begin{eqnarray}
A=A_W &:=& \genI^{m_s+1}\genII^{n_1}\cdots\genI^{m_{s-1}}\genII^{n_{s-1}},\nonumber\\
B=B_W &:=& \genI^{m_s}\genII^{n_1}\cdots\genI^{m_{s-1}}\genII^{n_{s-1}},\nonumber\\
C=C_W &:=& \genI^{m_s}\genII^{n_s}\cdots\genI^{m_{k}}\genII^{n_{k}},\nonumber
\end{eqnarray}
then
\begin{eqnarray}
 W=\begin{cases}\genI^{m_1-m_s-1}AC, & \mbox{if }n_1\geq n_s,\\
\genI^{m_1-m_s}BC, & \mbox{if }n_1<n_s.\end{cases}\label{WABC}
\end{eqnarray}
Routine arguments, that make use of inequalities satisfied by the exponents of $\genI$, may be used to show that the words $A$, $B$, $AC$ and $BC$ are all regular. Suppose $P=P_W$ and $Q=Q_W$ are the longest regular proper endings of $AC$ and $BC$, respectively. Since the exponents of $\genI$ in $A$ or $B$ [except $m_1$] are strictly less than the first exponent $m_s$ of $\genI$ in $C$, regularity requires that $P$ and $Q$ are subwords of $C$. Thus, there exist words $X=X_W$ and $Y=Y_W$ such that $C=XP$, or in the other case, $C=YQ$. The corresponding regular factorings are $AC=\lpar AX\rpar\regfac P$ and $BC=\lpar BY\rpar\regfac Q$. Consequently, the regular bracketing of $W$ is given by
\begin{eqnarray}
\lreg W\rreg=\begin{cases}\lpar\ad\genI\rpar^{m_1-m_s-1}\lpar\lbrak \lreg AX\rreg ,\lreg P\rreg\rbrak\rpar, & \mbox{if }n_1\geq n_s,\\
\lpar\ad\genI\rpar^{m_1-m_s}\lpar\lbrak \lreg BY\rreg,\lreg Q\rreg\rbrak\rpar, & \mbox{if }n_1<n_s.\end{cases}\label{bigWbrak}
\end{eqnarray}


\subsection{Inclusion compositions} Another important property of regular words involves interesting and useful Lie algebra manipulations when a regular subword is known. This will lead us to the notion of inclusion compositions that will be defined shorty. This notion is motivated by the following. 

\begin{proposition}[{\cite[Theorem 2.8.3(c)]{ufn95}}]\label{ORsubwordProp} If $V$ is a regular subword of the regular word $W=L\regfac R$, then either
\begin{enumerate}\item $V$ is a subword of $L$;
\item $V$ is a subword of $R$; or
\item there exists a proper ending $R'$ of $R$ such that $W=VR'$. In this case, we say that $V$ is a \emph{beginning of $W$ that intersects $R$}.
\end{enumerate}
\end{proposition}

\begin{proposition}\label{VRProp} If $V$ is a regular subword of a regular word $W$, then there exists a word $U$ such that $VU$ is regular, and that for some $k\in\N$, there exist regular words $U_1,U_2,\ldots,U_k$ and some $\varepsilon_1,\varepsilon_2,\ldots,\varepsilon_k\in\{-1,1\}$ such that if 
\begin{eqnarray}
\Phi:= \lpar\varepsilon_1\ad \lreg U_1\rreg\rpar\lpar\varepsilon_2\ad \lreg U_2\rreg\rpar\cdots \lpar\varepsilon_k\ad \lreg U_k\rreg\rpar,\nonumber
\end{eqnarray}
then
\begin{eqnarray}
\lreg W\rreg = \Phi\lpar\lreg VU\rreg\rpar.\label{VReq}
\end{eqnarray}
[For the case $k=0$, we interpret $\Phi$ as the identity map, or the empty composition of maps.]
\end{proposition}
\begin{proof} Let $W\in\allWords$ be regular. We use induction on $|W|$. Suppose that all words of length strictly less than $W$ satisfy the statement. If $W=L\regfac R$, then the inductive hypothesis applies to $L$ and $R$. We consider cases according to Proposition~\ref{ORsubwordProp}. If $V$ is a subword of $L$, then by the inductive hypothesis, there exists a word $S$ such that $VS$ is regular, and that for some regular words $S_1,S_2,\ldots,S_\ell$ and some $\eta_1,\eta_2,\ldots,$ $\eta_\ell\in\{-1,1\}$,
\begin{eqnarray}
\lreg L\rreg = \lpar\eta_1\ad \lreg S_1\rreg\rpar\lpar\eta_2\ad \lreg S_2\rreg\rpar\cdots \lpar\eta_k\ad \lreg S_k\rreg\rpar\lpar\lreg VS\rreg\rpar,\nonumber
\end{eqnarray}
and so,
\begin{eqnarray}
\lreg W\rreg  & = & \lreg L\regfac R\rreg = \lbrak \lreg L\rreg,\lreg R\rreg\rbrak,\nonumber\\
& = & \lbrak \lpar\eta_1\ad \lreg S_1\rreg\rpar\lpar\eta_2\ad \lreg S_2\rreg\rpar\cdots \lpar\eta_k\ad \lreg S_k\rreg\rpar\lpar\lreg VS\rreg\rpar,\lreg R\rreg\rbrak,\nonumber\\
& = &\lpar-1\cdot\ad \lreg R\rreg\rpar\lpar\eta_1\ad \lreg S_1\rreg\rpar\lpar\eta_2\ad \lreg S_2\rreg\rpar\cdots \lpar\eta_k\ad \lreg S_k\rreg\rpar\lpar\lreg VS\rreg\rpar,\nonumber
\end{eqnarray}
as desired. If $V$ is a subword of $R$, then by the inductive hypothesis, there exists a word $T$ such that $VT$ is regular, and that for some regular words $T_1,T_2,\ldots,T_n$ and some $\nu_1,\nu_2,\ldots,\nu_n\in\{-1,1\}$,
\begin{eqnarray}
\lreg R\rreg = \lpar\nu_1\ad \lreg T_1\rreg\rpar\lpar\nu_2\ad \lreg T_2\rreg\rpar\cdots \lpar\nu_k\ad \lreg T_k\rreg\rpar\lpar\lreg VT\rreg\rpar,\nonumber
\end{eqnarray}
and, consequently,
\begin{eqnarray}
\lreg W\rreg  & = & \lbrak \lreg L\rreg,\lreg R\rreg\rbrak,\nonumber\\
& = & \lbrak \lreg L\rreg, \lpar\nu_1\ad \lreg T_1\rreg\rpar\lpar\nu_2\ad \lreg T_2\rreg\rpar\cdots \lpar\nu_k\ad \lreg T_k\rreg\rpar\lpar\lreg VT\rreg\rpar\rbrak,\nonumber\\
& = & \lpar 1\cdot\ad \lreg L\rreg\rpar\lpar\nu_1\ad \lreg T_1\rreg\rpar\lpar\nu_2\ad \lreg T_2\rreg\rpar\cdots \lpar\nu_k\ad \lreg T_k\rreg\rpar\lpar\lreg VT\rreg\rpar,\nonumber
\end{eqnarray}
which is the desired form for $\lreg W\rreg$. The final case is when $V$ is a beginning of $W$ that intersects $R$. Here, $W=VX$ for some word $X$, and $\lreg W\rreg=\lreg VX\rreg$. Thus,\linebreak $\lreg W\rreg=\Phi\lpar\lreg VX\rreg\rpar$, where $\Phi$ is the identity map. This completes the proof.
\end{proof}

A statement similar to Proposition~\ref{VRProp} was briefly remarked in \cite[p. 38]{ufn95}, but we are aiming here for a more precise articulation of the statement, because it shall be crucial in a later definition. Also, the version in \cite[p. 38]{ufn95} does not make use of nested adjoint maps. [Recall Section~\ref{admapSec}.] In this author's opinion, the concept being expressed in \cite[p. 38]{ufn95} would be better comprehended or appreciated when expressed in terms of nested adjoint maps, just like how Proposition~\ref{VRProp} was stated above.

\begin{proposition}[{\cite[Theorem 2.8.5]{ufn95}}]\label{decompProp} For any nonempty $W\in\allWords$, there exists a unique finite sequence $V_1$, $V_2$, $\ldots$, $V_k$ of regular words such that $W=V_1V_2\cdots V_k$ and that $V_k\geq V_{k-1}\geq \cdots\geq V_1$. In this case, we say that $W=V_1V_2\cdots V_k$ is the \emph{regular decomposition} of the word $W$. In particular\footnote{The special case $k=1$ when $W$ is regular is not included in the statement of\linebreak \cite[Theorem 2.8.5]{ufn95}, but it can be found in the proof \cite[p. 35]{ufn95}. In this author's opinion, mentioning this special case, and even defining a term for it, aids in understanding the idea, especially because, in succeeding proofs, the concept will be used in very specific constructions.}, if $W$ is regular, then $k=1$, in which case the regular decomposition of $W$ is said to be \emph{trivial}.
\end{proposition}

The Lie subalgebra $\freeLie$ of $\freeAlg$ generated by $\{\genI,\genII\}$ is the free Lie algebra on $\{\genI,\genII\}$. That is, $\freeLie$ has the canonical universal property in the category of all Lie algebras over $\F$ with the same number of generators, or equivalently, that every Lie algebra generated by two elements is isomorphic to a quotient of $\freeLie$. The elements of $\freeLie$ are called the \emph{Lie polynomials in $\genI,\genII$}. The significance of regular words, and of the nonassociative regular words derived from them, is because of the following.

\begin{lemma}[{\cite[pp. 115]{shi09a}}]\label{ShirshovLem} The nonassociative regular words on $\genI$, $\genII$ form a basis for $\freeLie$.
\end{lemma}

The above result is attributed to A. I. Shirshov \cite[p. 2]{bok07}, because of the seminal paper \cite{shi09a}. However, the definition of a regular word by its "rotational" property is attributed to Lyndon, because of the classic paper \cite{che58}. But still, the significance of regular words and their regular bracketing in the basis theory for free Lie algebras definitely rests on the theorems and constructions on \cite{shi09a}. Thus, regular words are often referred to in the literature as \emph{Lyndon-Shirshov words}.

\begin{definition}\label{angleangleDef} With reference to the notation in Proposition~\ref{VRProp}, if $R\neq 1$, then suppose that $R$ has the regular decomposition $R=C_1C_2\cdots C_\ell$ according to Proposition~\ref{decompProp}. By $\lcomp W_V\rcomp$ we mean the Lie polynomial that results from replacing $\lreg VR\rreg$ in \eqref{VReq} by 
\begin{eqnarray}
\lpar-\ad \lreg C_\ell\rreg\rpar\lpar-\ad \lreg C_{\ell-1}\rreg\rpar\cdots \lpar-\ad \lreg C_1\rreg\rpar\lpar \lreg V\rreg\rpar.\nonumber
\end{eqnarray}
If $R=1$, then $\lreg VR\rreg=\lreg V\rreg$ is retained. In accordance with Lemma~\ref{ShirshovLem}, let the Lie polynomial $\lreg W\rreg-\lcomp W_V\rcomp$ be written as a (unique) linear combination of nonassociative regular words, as in
\begin{eqnarray}
\lreg W\rreg-\lcomp W_V\rcomp = \sum_{t=1}^n c_t \lreg B_t\rreg,\nonumber
\end{eqnarray}
where $c_1,c_2,\ldots, c_n\in\F\setdiff\{0\}$ and $B_1,B_2,\ldots,B_n$ are regular words such that\linebreak $B_n> B_{n-1}>\cdots> B_1 $. By the \emph{inclusion composition}\footnote{In the traditional theory of Gr\"obner-type bases for free Lie algebras, there is another type of composition called \emph{intersection composition}, which, together with the notion of inclusion composition, was originally developed in \cite{shi09b}. However, intersection compositions will play no role in the proofs of our main results.} of the regular word\footnote{The traditional theory also defines inclusion compositions in terms of linear combinations of nonassociative regular words. In this work we only consider linear combinations of exactly one nonassociative regular word. Consequently, we shall be dealing only with Lie algebras generated by two elements satisfying relations of the form $\lreg U\rreg=0$ where $U$ is a regular word.} $W$ with its subword $V$, we mean the Lie polynomial $\frac{1}{c_n}\lreg W\rreg-\frac{1}{c_n}\lcomp W_V\rcomp$. If $\frac{1}{c_n}\lreg W\rreg-\frac{1}{c_n}\lcomp W_V\rcomp=0$, then the inclusion composition of $W$ with $V$ is said to be \emph{trivial}.
\end{definition}

The ten inclusion compositions in the following lemma form the heart of this work.

\begin{lemma}\label{incompLem} If $k\in\N$ and $h,m,n\in\N\setdiff\{0\}$, then
\begin{eqnarray}
\lreg\genI^{m+2}\genII^n\rreg - \lcomp\genI^{m+2}\genII^n_{\genI^{m+1}\genII^n}\rcomp & = & 0,\label{IC1}\\
\lreg\genI^{h+k}\genII^m\genI^h\genII^n\rreg - \lcomp\genI^{h+k}\genII^m\genI^h\genII^n_{\genI^{h+\varepsilon}\genII^m}\rcomp & = & 0,\label{IC2}\\
\lreg\genI^{h+k}\genII^m\genI^h\genII^n\rreg - \lcomp\genI^{h+k}\genII^m\genI^h\genII^n_{\genI^h\genII^n}\rcomp & = & 0,\label{IC3}\\
-\lreg \genI^2\genII^{2}\rreg + \lcomp\genI^2\genII^{2}_{\genI^2\genII}\rcomp & = & 0,\label{IC4}\\
-\lreg \genI^2\genII^{n+1}\rreg + \lcomp\genI^2\genII^{n+1}_{\genI^2\genII^{n}}\rcomp & = & \lreg \genI\genII\genI\genII^n\rreg,\label{IC5}
\end{eqnarray}
where, in \eqref{IC2}, either $\varepsilon = 0$ (if $m<n$) or $\varepsilon =1$ (if $k\geq 1$, $m\geq n$). However, if $n>m+1$, then
\begin{eqnarray}
-\lreg\genI\genII^m\genI\genII^{n+1}\rreg +\lcomp\genI\genII^m\genI\genII^{n+1}_{\genI\genII^m\genI\genII^{n}}\rcomp & = & \lreg\genI\genII^{m+1}\genI\genII^{n}\rreg.\label{IC6}
\end{eqnarray}
With reference to Section~\ref{genregwordSec}, if $W=\genI^{m_1}\genII^{n_1}\genI^{m_2}\genII^{n_2}\cdots \genI^{m_k}\genII^{n_k}$ is an arbitrary regular word with $|W|\geq 2$, where the biggest value among $m_2,m_3,\ldots,m_k$ has its first occurrence at index $s$, and given the subwords
\begin{eqnarray}
A=A_W &=& \genI^{m_s+1}\genII^{n_1}\cdots\genI^{m_{s-1}}\genII^{n_{s-1}},\nonumber\\
B=B_W &=& \genI^{m_s}\genII^{n_1}\cdots\genI^{m_{s-1}}\genII^{n_{s-1}},\nonumber\\
C=C_W &=& \genI^{m_s}\genII^{n_s}\cdots\genI^{m_{k}}\genII^{n_{k}},\nonumber
\end{eqnarray}
such that $AC$ and $BC$ are regular, with $AC=\lpar AX\rpar\regfac P$ and $BC=\lpar BY\rpar\regfac Q$ for some words $X$ and $Y$, then
\begin{eqnarray}
\lreg W\rreg - \lcomp W_{AX}\rcomp & = & 0,\quad \mbox{if }n_1\geq n_s,\label{IC7}\\
\lreg W\rreg - \lcomp W_{BY}\rcomp & = & 0,\quad \mbox{if }n_1< n_s,\label{IC8}\\
\lreg W\rreg - \lcomp W_{P} \rcomp & = & 0,\quad \mbox{if }n_1\geq n_s,\label{IC9}\\
\lreg W\rreg - \lcomp W_{Q}\rcomp & = & 0,\quad \mbox{if }n_1< n_s.\label{IC10}
\end{eqnarray}
\end{lemma}
\begin{proof}\begin{enumerate}\item\label{incomp1} {\it Proof of \eqref{IC1}.} Using Example~\ref{Illus1Ex}\ref{mnform}, $\genI^{m+1}\genII^n$ is a regular subword of $\genI^{m+2}\genII^n$. In particular, if we consider Example~\ref{Illus2Ex}\ref{firstform2}, $\genI^{m+1}\genII^n$ is the longest regular proper ending of $\genI^{m+2}\genII^n$, and so, by Example~\ref{Illus3Ex}\ref{firstformBrak}, 
\begin{eqnarray}
\lreg\genI^{m+2}\genII^n\rreg = \lpar\ad\genI\rpar\lpar\lreg \genI^{m+1}\genII^n\rreg\rpar.\label{mainExEQ1}
\end{eqnarray}
From Definition~\ref{angleangleDef}, we find that in order to form the Lie polynomial\linebreak $\lcomp\genI^{m+2}\genII^n_{\genI^{m+1}\genII^n}\rcomp$, we simply retain $\lreg \genI^{m+1}\genII^n\rreg$ in the right-hand side of \eqref{mainExEQ1}. Thus, we have the trivial inclusion composition \eqref{IC1}.

\item\label{incomp2} {\it Proof of \eqref{IC2} and \eqref{IC3}.} We may simplify the regular bracketing shown in Example~\ref{Illus3Ex}\ref{genformBrak} as
\begin{eqnarray}
\lreg\genI^{h+k}\genII^m\genI^h\genII^n\rreg = \lpar\ad\genI\rpar^{k-\varepsilon}\lbrak\lreg\genI^{h+\varepsilon}\genII^m\rreg,\lreg\genI^h\genII^n\rreg\rbrak,\label{anothertriv0}
\end{eqnarray}
where
\begin{eqnarray}
\varepsilon=\begin{cases}0, &\mbox{if } m<n,\\
1, &\mbox{if }k\geq 1,\  m\geq n.\end{cases}\nonumber
\end{eqnarray}
Another perspective is that, if we let $R=\genI^h\genII^n$, then 
\begin{eqnarray}
\lreg\genI^{h+k}\genII^m\genI^h\genII^n\rreg = \lpar\ad\genI\rpar^{k-\varepsilon}\lpar\lreg\genI^{h+\varepsilon}\genII^mR\rreg\rpar,\label{anothertriv}
\end{eqnarray}
which is a form apparently more suitable in applying Proposition~\ref{VRProp} and Definition~\ref{angleangleDef} in determining $\lcomp\genI^{h+k}\genII^m\genI^h\genII^n_{\genI^{h+\varepsilon}\genII^m}\rcomp$. However, by Proposition~\ref{decompProp}, since $R$ is regular, its regular decomposition is trivial. Thus,\linebreak $\lreg\genI^{h+\varepsilon}\genII^mR\rreg$ in \eqref{anothertriv} is to be replaced by
\begin{eqnarray}
\lpar-\ad\lreg R\rreg\rpar\lpar\lreg\genI^{h+\varepsilon}\genII^m\rreg \rpar & = & -\lbrak\lreg R\rreg,\lreg\genI^{h+\varepsilon}\genII^m\rreg\rbrak,\nonumber\\
 & = &  \lbrak\lreg\genI^{h+\varepsilon}\genII^m\rreg,\lreg R\rreg\rbrak,\nonumber\\
& = & \lbrak\lreg\genI^{h+\varepsilon}\genII^m\rreg,\lreg\genI^h\genII^n\rreg\rbrak,\nonumber
\end{eqnarray}
which gives us the same thing as the right-hand side of \eqref{anothertriv0}. Thus, the inclusion composition of $\genI^{h+k}\genII^m\genI^h\genII^n$ with $\genI^{h+\varepsilon}\genII^m$ is trivial. Using arguments similar to those used in part \ref{incomp1} of this proof, the inclusion composition of $\genI^{h+k}\genII^m\genI^h\genII^n$ with $\genI^h\genII^n$ is also trivial. The result is \eqref{IC2} and \eqref{IC3}.

\item\label{incomp3} {\it Proof of \eqref{IC4} and \eqref{IC5}.} Following the third part of the proof of Proposition~\ref{VRProp}, for the regular word $\genI^2\genII^{n+1}$ and its regular subword $\genI^2\genII^{n}$, we find that $\lreg \genI^2\genII^{n+1}\rreg=\Phi\lpar\lreg \genI^2\genII^{n}U\rreg\rpar$ where $U=\genII$ and $\Phi$ is the identity map. Following Proposition~\ref{decompProp}, the regular decomposition of $U$ is simply $U=\genII$, and we further have
\begin{eqnarray}
\lreg \genI^2\genII^{n+1}\rreg & = & \Phi\lpar\lreg \genI^2\genII^{n}U\rreg\rpar,\nonumber\\
& = & \lreg \genI^2\genII^{n} \cdot\genII\rreg.\label{VRreplace}
\end{eqnarray}
By Definition~\ref{angleangleDef}, in order to form the Lie polynomial $\lcomp\genI^2\genII^{n+1}_{\genI^2\genII^{n}}\rcomp$, we replace $\lreg \genI^2\genII^{n} \cdot\genII\rreg$ in \eqref{VRreplace} by $\lpar-\ad\lreg\genII\rreg\rpar\lpar\genI^2\genII^{n}\rpar$. That is,
\begin{eqnarray}
\lcomp\genI^2\genII^{n+1}_{\genI^2\genII^{n}}\rcomp & = & \lbrak \lreg\genI^2\genII^{n}\rreg,\genII\rbrak,\nonumber\\
& = & \lbrak \lreg\genI\regfac\genI\genII^{n}\rreg,\genII\rbrak,\nonumber\\
& = & \lbrak \lbrak\genI,\lreg\genI\genII^{n}\rreg\rbrak,\genII\rbrak,\label{iJacobi}
\end{eqnarray}
while from Example~\ref{Illus3Ex}\ref{firstformBrak}, we obtain
\begin{eqnarray}
\lreg \genI^2\genII^{n+1}\rreg = \lbrak\genI,\lbrak\lreg\genI\genII^{n}\rreg,\genII\rbrak\rbrak.\label{iiJacobi}
\end{eqnarray}
We subtract \eqref{iJacobi} from \eqref{iiJacobi}, and after routine computations that make use of the Jacobi identity and the skew-symmetry of the Lie bracket, we obtain the inclusion composition
\begin{eqnarray}
-\lreg \genI^2\genII^{n+1}\rreg + \lcomp\genI^2\genII^{n+1}_{\genI^2\genII^{n}}\rcomp = \lbrak\lreg \genI\genII\rreg,\lreg\genI\genII^n\rreg\rbrak,\nonumber
\end{eqnarray}
which reduces to the trivial inclusion composition \eqref{IC4} if $n=1$, but if $n\geq 2$, then using Example~\ref{Illus3Ex}\ref{firstformBrak}--\ref{secondformBrak}, we get \eqref{IC5}.

\item\label{incomp4} {\it Proof of \eqref{IC6}.} From $m<m+1<n<n+1$, we get $m<n+1$. By\linebreak Example~\ref{Illus1Ex}\ref{mnmnform}, the inequalities $m+1<n$ and $m<n+1$ imply that the words $\genI\genII^{m+1}\genI\genII^{n}$ and $\genI\genII^m\genI\genII^{n+1}$ are regular, and using arguments and computations similar to those used in part \ref{incomp3} of this proof, we get \eqref{IC6}.

\item\label{incomp5} {\it Proof of \eqref{IC7}--\eqref{IC10}.} The regular bracketing \eqref{bigWbrak} of $W$ may be rewritten in two other ways:
\begin{eqnarray}
\lreg W\rreg  & = & \begin{cases}\lpar\ad\genI\rpar^{m_1-m_s-1}\lpar -\ad \lreg P\rreg\rpar\lpar\lreg AX\rreg\rpar, & \mbox{if }n_1\geq n_s,\\
\lpar\ad\genI\rpar^{m_1-m_s}\lpar-\ad \lreg Q\rreg\rpar\lpar\lreg BY\rreg\rpar, & \mbox{if }n_1<n_s,\end{cases}\nonumber\\
\lreg W\rreg  & = & \begin{cases}\lpar\ad\genI\rpar^{m_1-m_s-1}\lpar\ad \lreg AX\rreg\rpar\lpar \lreg P\rreg\rpar, & \mbox{if }n_1\geq n_s,\\
\lpar\ad\genI\rpar^{m_1-m_s}\lpar \ad \lreg BY\rreg\rpar\lpar\lreg Q\rreg\rpar, & \mbox{if }n_1<n_s.\end{cases}\nonumber
\end{eqnarray}
Following Proposition~\ref{VRProp} and Definition~\ref{angleangleDef}, the above equations imply that the inclusion compositions in the left-hand sides of \eqref{IC7}--\eqref{IC10} are indeed trivial.\qedhere
\end{enumerate}
\end{proof}

\begin{lemma}\label{incompLem2} Let $W$ and $V$ be regular words such that $V$ is a subword of $W$, and let $\anyLIdeal$ be a Lie ideal of $\freeLie$.
\begin{enumerate}\item\label{LIdeal1} If the inclusion composition of $W$ with $V$ is trivial and $\lreg V\rreg\in\anyLIdeal$, then\linebreak $\lreg W\rreg\in\anyLIdeal$.
\item\label{LIdeal2} If $\lreg V\rreg\in\anyLIdeal$ and $\lreg W\rreg\in\anyLIdeal$, then the inclusion composition of $W$ with $V$ is an element of $\anyLIdeal$.
\end{enumerate}
\end{lemma}
\begin{proof} By Proposition~\ref{VRProp} and Definition~\ref{angleangleDef},
{\scriptsize\begin{eqnarray}
\lcomp W_V\rcomp = \lpar\varepsilon_1\ad \lreg U_1\rreg\rpar\cdots \lpar\varepsilon_k\ad \lreg U_k\rreg\rpar\lpar-\ad \lreg C_\ell\rreg\rpar\lpar-\ad \lreg C_{\ell-1}\rreg\rpar\cdots \lpar-\ad \lreg C_1\rreg\rpar\lpar \lreg V\rreg\rpar,\label{longAd}
\end{eqnarray}}

\noindent where $\varepsilon_i\in\{-1,1\}$ and $U_i,C_j\in\freeLie$ for any $i\in\{1,2,\ldots,k\}$ and any\linebreak $j\in\{1,2,\ldots,\ell\}$. Since $\anyLIdeal$ is a Lie ideal of $\freeLie$, if $\lreg V\rreg\in\anyLIdeal$, then by \eqref{longAd}, $\lcomp W_V\rcomp\in\anyLIdeal$. If the inclusion composition of $W$ with $V$ is trivial, then $\lreg W\rreg=\lcomp W_V\rcomp\in\anyLIdeal$, proving \ref{LIdeal1}. If instead we have $\lreg V\rreg,\lreg W\rreg\in\anyLIdeal$, then by \eqref{longAd}, $\lcomp W_V\rcomp$ and $\lreg W\rreg$ are elements of $\anyLIdeal$, and so is any linear combination of them, such as the inclusion composition of $W$ with $V$. This proves \ref{LIdeal2}. 
\end{proof}


\subsection{A Lie ideal of $\freeLie$ and its normal complement}

As according to Lemma~\ref{ShirshovLem}, the nonassociative regular words form a basis for\linebreak $\freeLie$, and at this point, we partition this basis into two kinds: what shall be important in the subsequent development of Lie structure theory for the Heisenberg-Weyl algebra are the nonassociative regular words
\begin{eqnarray}
\qquad\qquad\genII,\qquad\lreg \genI\genII^n\rreg,\qquad\qquad\qquad (n\in\N),\label{HWnormalLie}
\end{eqnarray}
and so if 
\begin{eqnarray}
\normalLie:=\mbox{Span}_{\F}\  \{\genII,\  \lreg \genI\genII^n\rreg\  :\  n\in\N\},\nonumber
\end{eqnarray}
then all the nonassociative regular words \emph{not} in \eqref{HWnormalLie} span a vector subspace $\normalLie^c$ of $\freeLie$ such that we have the direct sum decomposition 
\begin{eqnarray}
\freeLie=\normalLie\oplus\normalLie^c.\label{TheDsum}
\end{eqnarray}
Later we shall need to classify the aforementioned basis elements of $\normalLie^c$, and for this we need the necessary condition for the regularity of a word from Lemma~\ref{badegLem}.

By some routine argument, the set of all nonassociative regular words may be partitioned using the equivalence relation under which two nonassociative regular words $\lreg W\rreg$ and $\lreg V\rreg$ are related if and only if 
\begin{eqnarray}
\degBA W=\degBA V.\label{equivrel}
\end{eqnarray}
[Recall how this number was defined in \eqref{degBAdef0}--\eqref{degBAdef}.] The equivalence class that contains all $\lreg W\rreg$ with $\deg_{\genII\genI}W=0$ is precisely the set containing the basis elements of $\normalLie$ from \eqref{HWnormalLie} together with
\begin{eqnarray}
\lreg\genI^{m+1}\genII^n\rreg, & &  m,n\in\N\setdiff\{0\}.\label{compbasis2} 
\end{eqnarray}
Consequently, all nonassociative regular words $\lreg W\rreg$ with $\deg_{\genII\genI}W\geq 2$, together with those in \eqref{compbasis2}, form a basis, which we shall refer to as the \emph{regular basis}, for $\normalLie^c$. 

\begin{lemma}\label{TheBigLem} Every regular basis element of $\normalLie^c$ is contained in the Lie ideal of $\freeLie$ generated by
\begin{eqnarray}
\qquad\qquad\lreg \genI^2\genII^n\rreg,\qquad\qquad\qquad (n\in\N\setdiff\{0\}).\label{twoLIdealGens}
\end{eqnarray}
\end{lemma}
\begin{proof} Let $\lreg W\rreg$ be a regular basis element of $\normalLie^c$, and let $\twoLIdeal$ be the Lie ideal of $\freeLie$ generated by \eqref{twoLIdealGens}. This proof is organized according to the equivalence class, under the equivalence relation defined by \eqref{equivrel}, to which $\lreg W\rreg$ belongs. In each case, we shall be using an inclusion composition from Lemma~\ref{incompLem}, and then Lemma \ref{incompLem2}, to produce the desired set membership $\lreg W\rreg\in\twoLIdeal$.

If $\degBA W=0$, then $\lreg W\rreg$ is either one of \eqref{HWnormalLie}, or one of \eqref{compbasis2}, where the former are not regular basis elements of $\normalLie^c$, while the latter are. Equivalently, $W$ is a product of a power of $\genI$ followed by a power of $\genII$, where both exponents are positive, but that of $\genI$ is at least $2$. If this exponent of $\genI$ is exactly $2$, then $\lreg W\rreg$ is one of the generators \eqref{twoLIdealGens} of $\twoLIdeal$, and we are done. We proceed by induction. If, for some positive integer $m$, we have $\lreg\genI^{m+1}\genII^n \rreg\in\twoLIdeal$, then we simply use the trivial inclusion composition \eqref{IC1} and Lemma~\ref{incompLem2}\ref{LIdeal1}, to deduce $\lreg\genI^{m+2}\genII^n \rreg\in\twoLIdeal$. By induction, we get the desired result.

For the case $\degBA W =1$, we have $W=\genI^{h+k}\genII^m\genI^h\genII^n$ for some $h,k,m,n\in\N$ with $h,m,n$ positive. If $h\geq 2$, then, given $\varepsilon\in\{0,1\}$ from \eqref{IC2}, both $h$ and $h+\varepsilon$ are at least $2$, and we have $\lreg\genI^{h+\varepsilon}\genII^m\rreg,\lreg\genI^h\genII^n\rreg\in\twoLIdeal$, according to the previous case.\linebreak Using the trivial inclusion composition \eqref{IC2} or \eqref{IC3}, and Lemma~\ref{incompLem2}\ref{LIdeal1},\linebreak $\lreg\genI^{h+k}\genII^m\genI^h\genII^n\rreg\in\twoLIdeal$. The trivial inclusion composition \eqref{IC2} may also be used for the subcase $h=1$ and $\varepsilon=1$. We now consider the subcase $h=1$ and $k=0$. That is, $\genI^{h+k}\genII^m\genI^h\genII^n=\genI\genII^m\genI\genII^n$. We use induction on $m$. If $m=1$, then we use the inclusion compostion \eqref{IC5} where $\lreg\genI^2\genII^{n+1}\rreg,\lreg \genI^2\genII^n\rreg\in\twoLIdeal$. By Lemma~\ref{incompLem2}\ref{LIdeal2}, $\lreg\genI\genII\genI\genII^n\rreg\in\twoLIdeal$. Suppose that for some positive integer $m$, for any integer $n>m$, $\lreg\genI\genII^m\genI\genII^n\rreg\in\twoLIdeal$. To proceed with the inductive step at $m+1$, we assume that $n>m+1$ so that, by Example~\ref{Illus1Ex}\ref{mnmnform}, $\genI\genII^{m+1}\genI\genII^n$ is regular. From $n>m+1$, we get $n+1>n>m+1>m$. Thus, both $n+1>m$ and $n>m$ are true. By the inductive hypothesis,\linebreak $\lreg\genI\genII^{m}\genI\genII^{n+1}\rreg,\  \lreg\genI\genII^{m}\genI\genII^{n}\rreg\in\twoLIdeal$, and using the inclusion composition \eqref{IC6} and\linebreak Lemma~\ref{incompLem2}\ref{LIdeal2}, we obtain $\lreg\genI\genII^{m+1}\genI\genII^n\rreg\in\twoLIdeal$, which completes the induction, and also, the proof for the case $\degBA W =1$.

We now consider the case $\degBA W\geq 1$, and we use induction on $\degBA W$. First, we recall the notation in Section~\ref{genregwordSec} and Lemma~\ref{incompLem}:
\begin{eqnarray}
W=\genI^{m_1}\genII^{n_1}\genI^{m_2}\genII^{n_2}\cdots \genI^{m_k}\genII^{n_k},\label{necessW2}
\end{eqnarray}
for some positive integers $m_1,n_1,m_2,n_2,\ldots,m_k,n_k$, where $\degBA W=k-1$. If the biggest value among $m_2,m_3,\ldots,m_k$ has its first occurrence at index $s$, then $m_1\geq m_s$, and we consider subwords of $W$ according to the position of $\genI^{m_s}\genII^{n_s}$: \begin{eqnarray}
A=A_W &=& \genI^{m_s+1}\genII^{n_1}\cdots\genI^{m_{s-1}}\genII^{n_{s-1}},\nonumber\\
B=B_W &=& \genI^{m_s}\genII^{n_1}\cdots\genI^{m_{s-1}}\genII^{n_{s-1}},\nonumber\\
C=C_W &=& \genI^{m_s}\genII^{n_s}\cdots\genI^{m_{k}}\genII^{n_{k}},\nonumber
\end{eqnarray}
where $AC$ and $BC$ are regular, with 
\begin{eqnarray}
AC=\lpar AX\rpar\regfac P,\label{regfacProof1}\\
BC=\lpar BY\rpar\regfac Q,\label{regfacProof2}
\end{eqnarray}
for some words $X$ and $Y$. Also, we may rewrite \eqref{WABC} as
\begin{eqnarray}
 W=\begin{cases}\genI^{m_1-m_s-1}AXP, & \mbox{if }n_1\geq n_s,\\
\genI^{m_1-m_s}BYQ, & \mbox{if }n_1<n_s.\end{cases}\label{breakupW}
\end{eqnarray}
By Lemma~\ref{badegLem}, the regular words $AX$, $P$, $BY$ and $Q$, which appear in the regular factorings \eqref{regfacProof1},\eqref{regfacProof2}, may also be expressed in the form \eqref{necessW2}. That is, for each\linebreak $U\in\{AX,P,BY,Q\}$, there exists a word $\overline{U}$ such that $U=\genI\overline{U}\genII$. A routine argument may be used to show that $\degBA U=\degBA\overline{U}$. Also, \eqref{breakupW} may be rewritten as
\begin{eqnarray}
 W=\begin{cases}\genI^{m_1-m_s-1+1}\overline{AX}\cdot\genII\genI\cdot\overline{P}\genII, & \mbox{if }n_1\geq n_s,\\
\genI^{m_1-m_s+1}\overline{BY}\cdot\genII\genI\cdot\overline{Q}\genII, & \mbox{if }n_1<n_s,\end{cases}\nonumber
\end{eqnarray}
which show that one occurrence of $\genII\genI$ exists between the words $\overline{AX}$ and $\overline{P}$, and also between $\overline{BY}$ and $\overline{Q}$. Consequently, for each $U\in\{AX,P,BY,Q\}$,
\begin{eqnarray}
\degBA U = \degBA\overline U <\degBA W.\label{inductdegBA}
\end{eqnarray}
We now proceed with the use of induction on $\degBA W\geq 1$ to prove that $\lreg W\rreg\in\twoLIdeal$. Suppose that for any positive integer $t<\degBA W$, any regular word $U$ with $\degBA U =t$ has the property that $\lreg U\rreg\in\twoLIdeal$. Define $\widehat{W}$ as either $AX$ or $P$ if $n_1\geq n_s$, or as either $BY$ or $Q$ if $n_1<n_s$. If $1\leq \degBA \widehat{W}$, then, in conjunction with \eqref{inductdegBA} and the inductive hypothesis, $\lreg\widehat{W}\rreg\in\twoLIdeal$. Using whichever of the trivial inclusion compositions \eqref{IC7}--\eqref{IC10} is appropriate, we obtain
\begin{eqnarray}
\lreg W\rreg -\lcomp W_{\widehat{W}}\rcomp =0,\nonumber
\end{eqnarray}
and by Lemma~\ref{incompLem2}\ref{LIdeal1}, $\lreg W\rreg\in\twoLIdeal$. This does not cover the case when $\degBA\widehat{W}=0$ for any $\widehat{W}\in\{AX,P,BY,Q\}$. Suppose that we are indeed in such a case. If $n_1\geq n_s$, then $\widehat{W}$ can only be $AX$ or $P$, both of which we assume to have zero occurrence of $\genII\genI$. Thus, there exist $\mu_1,\nu_1,\mu_2,\nu_2\in\N$ such that $AX=\genI^{\mu_1}\genII^{\nu_1}$, $P=\genI^{\mu_2}\genII^{\nu_2}$ so that from \eqref{breakupW},
\begin{eqnarray}
W = \genI^{m_1-m_s-1+\mu_1}\genII^{\nu_1}\genI^{\mu_2}\genII^{\nu_2}.\label{trivialW}
\end{eqnarray}
A routine argument may be used to show that the regularity of $W$ implies $\nu_2\neq 0$. But then, we see from \eqref{trivialW} that $\degBA W = 0$, if one of $\nu_1$, $\mu_2$ is zero, or $\degBA W = 1$, if $\nu_1$, $\mu_2$ are both nonzero, and both cases have already been dealt with earlier. An analogous argument may be used for choices of $\widehat{W}$ when $n_1< n_s$, and the proof is complete.
\end{proof}

\begin{theorem}\label{QuotientThm} Let $\twoLie$ be a Lie algebra generated by two elements $g_1$, $g_2$. Let\linebreak $\twoGen:\{\genI,\genII\}\into\twoLie$ be the function defined by $\twoGen:\genI\mapsto g_1,\  \genII\mapsto g_2$. If\linebreak $\id:\{\genI,\genII\}\into\freeLie$ is the identity map, let $\canmap:\freeLie\into\twoLie$ be the canonical Lie algebra homomorphism, that is, the unique Lie algebra homomorphism for which $\twoGen=\canmap\circ\id$. Suppose further that under $\canmap$,
\begin{enumerate}
\item\label{notComptonotKer}  the images of the basis elements
\begin{eqnarray}
\qquad\qquad\genII,\qquad\lreg \genI\genII^n\rreg,\qquad\qquad\qquad (n\in\N),\nonumber
\end{eqnarray}
of $\normalLie$, from \eqref{HWnormalLie}, form a basis for $\twoLie$, and
\item\label{IdtoKer} for any $n\in\N$, $\canmap(\genI)$ commutes with $\canmap\lpar\lreg\genI\genII^n\rreg\rpar$, 
\end{enumerate}
then $\ker\canmap$ is generated by
\begin{eqnarray}
\qquad\qquad\lreg \genI^2\genII^n\rreg,\qquad\qquad\qquad (n\in\N\setdiff\{0\}).\label{LIdealgens}
\end{eqnarray}
\end{theorem}
\begin{proof} Let $\twoLIdeal$ be the Lie ideal of $\freeLie$ generated by \eqref{LIdealgens}. Using the assumption \ref{IdtoKer}, and also Example~\ref{Illus2Ex}\ref{firstform2}, for any $n\in\N\setdiff\{0\}$,
\begin{eqnarray}
0 & = & \lbrak\canmap(\genI),\canmap\lpar\lreg\genI\genII^n\rreg\rpar\rbrak=\canmap\lpar\lbrak\genI,\lreg\genI\genII^n\rreg\rbrak\rpar,\nonumber\\
& = & \canmap\lpar\lbrak\lreg\genI\rreg,\lreg\genI\genII^n\rreg\rbrak\rpar=\canmap\lpar\lreg\genI\regfac\genI\genII^n\rreg\rpar,\nonumber\\
& = & \canmap\lpar\lreg\genI^2\genII^n\rreg\rpar.\nonumber
\end{eqnarray}
Thus, every generator of $\twoLIdeal$ is in $\ker\canmap$, but from Lemma~\ref{TheBigLem}, every regular basis element of $\normalLie^c$ is in $\twoLIdeal$. Hence,
\begin{eqnarray}
\normalLie^c\subseteq\twoLIdeal\subseteq\ker\canmap. \label{TheSetInclusions}
\end{eqnarray}
Suppose $f\in\freeLie$ such that $f\notin\normalLie^c$. By the direct sum decomposition \eqref{TheDsum}, there exist $$c_1,c_2,\ldots,c_k,e_1,e_2,\ldots,e_\ell\in\F$$ (with at least one $c_I\neq 0$) such that
\begin{eqnarray}
f = \sum_{i=1}^k c_i A_i + \sum_{j=1}^\ell e_j B_j,\label{whatisf}
\end{eqnarray}
where,  for any $i$, $A_i$ is one of the basis elements of $\normalLie$ in \eqref{HWnormalLie}, and $B_j\in\normalLie^c\subseteq\ker\canmap$ for all $j$. Appyling $\canmap$ to both sides of \eqref{whatisf}, the second summation vanishes, and so $\canmap(f)=\displaystyle\sum_{i=1}^k c_i \canmap\lpar A_i\rpar$. Tending towards a contradiction, suppose $f\in\ker\canmap$. Then\linebreak $0=\displaystyle\sum_{i=1}^k c_i \canmap\lpar A_i\rpar$ where one of the scalars $c_I$ is nonzero, but according to the assumption \ref{notComptonotKer}, $\canmap\lpar A_1\rpar,\canmap\lpar A_2\rpar,\ldots,\canmap\lpar A_k\rpar$ are linearly independent, a contradiction. Hence, $f\notin\ker\canmap$. We have thus proven $f\notin\normalLie^c$ implies $f\notin\ker\canmap$, and we may augment \eqref{TheSetInclusions} into
\begin{eqnarray}
\normalLie^c\subseteq\twoLIdeal\subseteq\ker\canmap\subseteq\normalLie^c.\nonumber
\end{eqnarray}
Therefore, $\ker\canmap=\twoLIdeal$.
\end{proof}

\section{Some Lie structure theorems}\label{ConcreteSec}

The defining relation $AB=BA+1$ for the Heisenberg-Weyl algebra may be used to replace any occurence of $AB$ in a word on $\{A, B\}$, by $BA+1$. After using the distributivity laws, the new linear combination of words on $\{A,B\}$ may be checked for any occurence of $AB$, which again may be replaced by $BA+1$. This process terminates, and the result is a linear combination of words $W$ on $\{A, B\}$ with $\deg_{AB}W=0$. That this process indeed terminates is guaranteed by the Diamond Lemma for Ring Theory \cite[Theorem 2.1]{ber78}. More precisely, the Diamond Lemma may be used to show that the elements
\begin{eqnarray}
B^mA^n,\qquad\qquad (m,n\in\N),\label{thebasis}
\end{eqnarray}
form a basis for $\HWeyl$. Given $m,n\in\N$, the relation
\begin{eqnarray}
A^nB^m = \sum_{k=0}^{\min\{m,n\}}{m\choose k}{n\choose k}k!B^{m-k}A^{n-k},\label{AnBm}
\end{eqnarray}
may be used to rewrite the product of any two basis elements from \eqref{thebasis} as a linear combination of \eqref{thebasis}. That is, the formula \eqref{AnBm} may be used to compute the structure constants of the algebra $\HWeyl$. One of the earliest appearances of the formula \eqref{AnBm} in the literature is \cite[Equation (11)]{sha79}, which has an operator-theoretic proof.

Since the defining relation for $\HWeyl$ is equivalent to $B(-A)=(-A)B+1$, there exists an algebra homomorphism $\isoH:\HWeyl\into\HWeyl$ such that
\begin{eqnarray}
\isoH &:& A\mapsto B,\quad B\mapsto -A.\label{HWisomap}
\end{eqnarray}
Using \eqref{AnBm}, each basis element of $\HWeyl$ from \eqref{thebasis} is the image under $\isoH$ of some element of $\HWeyl$, and so by some routine arguments, $\isoH$ is surjective. Also, 
\begin{eqnarray}
\isoH^2 &:& A\mapsto -A,\quad B\mapsto -B,\nonumber\\
\isoH^3 &:& A\mapsto -B,\quad B\mapsto A,\nonumber
\end{eqnarray}
where by exponentiation, we mean function composition of $\isoH$ with itself. Consequently, $\isoH^3$ serves as inverse for $\isoH$. Thus, $\isoH$ is an isomorphism. The idea that such an isomorphism exists had one of its first appearances also in the paper \cite{sha79}, but was not articulated in algebraic terms, and which is instead based on the vague idea of ``substuting'' for $A$ and $B$ some other objects which, in our description above, is equivalent to $\isoH(A)$ and $\isoH(B)$ \cite[Equation (12)]{sha79}.

The \emph{Heisenberg-Weyl Lie algebra} is the Lie algebra $\LieHWeyl$ generated by two elements $X$ and $Y$ satisfying the relation $XY-YX=1$. We immediately find that $\LieHWeyl$ is isomorphic to the Lie subalgebra of $\HWeyl$ generated by $A$ and $B$. Also, routine verification shows that $\LieHWeyl$ is three-dimensional, and is in fact, one of the classical low-dimensional Lie algebras. Thus, the algebra generators $A$ and $B$ are not able to generate the whole Lie algebra $\HWeyl$. It turns out that two additional generators are needed.

\begin{theorem}\label{GenThm} As a Lie algebra, $\HWeyl$ is generated by $A$, $B$, $BA^2$, $B^2A$.
\end{theorem}
\begin{proof} Let $\subHWeyl$ be the Lie subalgebra of $\HWeyl$ generated by $A$, $B$, $BA^2$, $B^2A$. Thus, $\subHWeyl\subseteq\HWeyl$, and we only need to show $\HWeyl\subseteq\subHWeyl$, but this reduces to showing that every basis element in \eqref{thebasis} is in $\subHWeyl$. Concerning those basis elements $B^sA^t$ where exactly one of $s$ or $t$ is zero, induction and the relation \eqref{AnBm} may be used to show that
\begin{eqnarray}
B^{m+1} & = &\frac{1}{m!}(\ad B^2A)^m(B)\in\subHWeyl,\\
A^{m+1} & = &\frac{1}{m!}(-\ad BA^2)^m(A)\in\subHWeyl,\label{coreLieclosure}
\end{eqnarray}
for any $m\in\N$. We now consider those basis elements $B^sA^t$ where $s$ and $t$ are both nonzero, or are both zero. We use induction on $s+t$. The smallest possibility is $s+t=0$, and by the defining relation of $\HWeyl$, $B^0A^0=1=\lbrak A,B\rbrak\in\subHWeyl$. Suppose that any basis element $B^iA^j$ with $i+j<s+t$ are elements of $\subHWeyl$. By routine computations that make use of \eqref{AnBm},
\begin{eqnarray}
B^sA^t & = &\frac{1}{(s+1)(t+1)}\lbrak B^{s+1},A^{t+1}\rbrak\nonumber\\ && +\sum_{k=2}^{\min\{s+1,t+1\}}{{s+1}\choose k}{{t+1}\choose k}\frac{k!}{(s+1)(t+1)}B^{s+1-k}A^{t+1-k}.\qquad\label{BsAt}
\end{eqnarray}
But by the previous case, $B^{s+1}$ and $A^{t+1}$ are elements of $\subHWeyl$, and so is their Lie bracket. The inductive hypothesis also guarantees that $B^{s+1-k}A^{t+1-k}\in\subHWeyl$ for all\linebreak $k\in\{2,3,\ldots,\min\{s+1,t+1\}\}$. Thus,  we find from \eqref{BsAt} that $B^sA^t\in\subHWeyl$. This completes the proof.
\end{proof}

From this point onward, we assume that the characteristic of the field $\F$ is zero.

\begin{lemma}\label{coreLieLem} If $\coreLie$ is the linear span of
\begin{eqnarray}
BA^2,\quad A^n,\qquad (n\in\N\setdiff\{0\}),\label{coreLiebasis}
\end{eqnarray}
then $\coreLie$ is a Lie subalgebra of $\HWeyl$, with a presentation by generators $\ImgenI$, $\ImgenII:=B^2A$, and relations
\begin{eqnarray}
\lpar\ad\ImgenI\rpar\lpar-\ad\ImgenII\rpar^n\lpar\ImgenI\rpar=0,\qquad\qquad\qquad (n\in\N\setdiff\{0\}).\label{bargens}
\end{eqnarray}
Furthermore, the Lie subalgebras of $\coreLie$ (other than $\coreLie$ itself) in the lower central series are given by
\begin{eqnarray}
\lpar\ad\coreLie\rpar^k\lpar\coreLie\rpar & = & \kcoreLieDer,\nonumber
\end{eqnarray}
for all $k\in\N\setdiff\{0\}$. Hence, $\coreLie$ is non-nilpotent but solvable.
\end{lemma}
\begin{proof} The spanning elements \eqref{coreLiebasis} of $\coreLie$ are among the basis elements \eqref{thebasis} of $\HWeyl$. Thus, the spanning elements \eqref{coreLiebasis} are linearly independent, and hence form a basis for $\coreLie$. Let $F,G$ be any two of said basis elements. If both $F$ and $G$ are powers of $A$, or if $F=BA^2=G$, then $\lbrak F,G\rbrak=0\in\coreLie$. If one of $F$ or $G$ is $BA^2$ and the other is a power of $A$, then there exist $\varepsilon\in\{-1,1\}$ and $n\in\N\setdiff\{0\}$ such that, by the skew-symmetry of the Lie bracket, $\lbrak F,G\rbrak=\varepsilon\lbrak A^n,BA^2\rbrak$, where, with the use of \eqref{AnBm}, routine computations may be used to show $\lbrak A^n,BA^2\rbrak=nA^{n+1}\in\coreLie$. Thus, $\lbrak F,G\rbrak\in\coreLie$. At this point, we have shown that $\coreLie$ is closed under the Lie bracket, and is hence a Lie subalgebra of $\HWeyl$.

If $\coreLie_0$ is the Lie subalgebra of $\HWeyl$ generated by $A$ and $BA^2$, then the fact that $\coreLie$ is a Lie algebra and that $A,BA^2\in\coreLie$ imply $\coreLie_0\subseteq\coreLie$. The other set inclusion follows from the fact that, by \eqref{coreLieclosure}, every basis element of $\coreLie$ is in $\coreLie_0$. Hence, $\coreLie_0=\coreLie$.

Let $\HWsetmap:\{\genI,\genII\}\into\coreLie$ be defined by $\HWsetmap:\genI\mapsto A,\  \genII\mapsto BA^2$. Given the identity map $\id:\{\genI,\genII\}\into\freeLie$, let $\HWCanmap:\freeLie\into\coreLie$ be the canonical Lie algebra homomorphism, for which, $\HWsetmap = \HWCanmap\circ\id$. Thus, for each $n\in\N$,
\begin{eqnarray}
\HWCanmap(\genII)  & = & B^2A,\label{ImBasis1}\\
\HWCanmap\lpar\lreg\genI\genII^n\rreg\rpar & = &\HWCanmap\lpar\lpar-\ad\genII\rpar^n(\genI)\rpar,\nonumber
\end{eqnarray}
and since $\HWCanmap$ is a Lie algebra homomorphism,
\begin{eqnarray}
\HWCanmap\lpar\lreg\genI\genII^n\rreg\rpar & = &\lpar\lpar-\ad\HWCanmap(\genII)\rpar^n\lpar\HWCanmap(\genI)\rpar\rpar,\nonumber\\
& = & \lpar-\ad B^2A\rpar^n(A),\nonumber
\end{eqnarray}
and by \eqref{coreLieclosure},
\begin{eqnarray}
\HWCanmap\lpar\lreg\genI\genII^n\rreg\rpar = n!\  A^{n+1},\quad(n\in\N).\label{ImBasis2}
\end{eqnarray}
Since the characteristic of $\F$ is assumed to be zero, the factorials in \eqref{ImBasis2} are nonzero, and a routine argument may be used to show that, by \eqref{ImBasis1},\eqref{ImBasis2}, the images, under the canonical map $\HWCanmap$, of 
\begin{eqnarray}
\qquad\qquad\genII,\qquad\lreg \genI\genII^n\rreg,\qquad\qquad\qquad (n\in\N),\nonumber
\end{eqnarray}
form a basis for $\coreLie$. Also, for each $n\in\N$,
\begin{eqnarray}
\lbrak\HWCanmap(\genI),\HWCanmap\lpar \lreg\genI\genII^n\rreg\rpar\rbrak = \lbrak A,n!\  A^{n+1}\rbrak=0,\nonumber
\end{eqnarray}
or equivalently, $\HWCanmap(\genI)$ commutes with $\HWCanmap\lpar \lreg\genI\genII^n\rreg\rpar$. At this point, we have shown that all hypotheses in Theorem~\ref{QuotientThm} are true. Thus, $\ker\HWCanmap$ is generated by
\begin{eqnarray}
\qquad\qquad\lreg \genI^2\genII^n\rreg,\qquad\qquad\qquad (n\in\N\setdiff\{0\}),\nonumber
\end{eqnarray}
or equivalently, the quotient Lie algebra $\freeLie/\ker\HWCanmap$ has a presentation by generators $\genI$, $\genII$ and relations which assert that for each $n\in\N\setdiff\{0\}$,
\begin{eqnarray}
\lreg \genI^2\genII^n\rreg=0,\nonumber
\end{eqnarray}
which, by Example~\ref{Illus3Ex}\ref{firstformBrak}, has the equivalent form
\begin{eqnarray}
\lpar\ad\genI\rpar\lpar-\ad\genII\rpar^n\lpar\genI\rpar=0,\label{prerelation}
\end{eqnarray}
but since $\HWCanmap$ is a Lie algebra homomorphism $\freeLie\into\coreLie$, we have\linebreak $\coreLie\cong\freeLie/\ker\HWCanmap$, and so, $\coreLie$ has a presentation by generators $A=\HWCanmap(\genI)$,\linebreak $\Omega:=B^2A=\HWCanmap(\genII)$ and relations, similar in form to \eqref{prerelation}, which assert that, for each $n\in\N\setdiff\{0\}$,
\begin{eqnarray}
\lpar\ad\HWCanmap(\genI)\rpar\lpar-\ad\HWCanmap(\genII)\rpar^n\lpar\HWCanmap(\genI)\rpar & = & 0,\nonumber\\
\lpar\ad A\rpar\lpar-\ad\Omega\rpar^n\lpar A\rpar & = & 0.\nonumber
\end{eqnarray}

We now prove $\lpar\ad\coreLie\rpar^k\lpar\coreLie\rpar =\mathcal{G}_k:= \kcoreLieDer$, or equivalently $\lbrak\coreLie,\mathcal{G}_{k-1}\rbrak=\mathcal{G}_k$, by induction. Let $k\in\N\setdiff\{0\}$. Suppose that for any positive integer $t<k$, $\lpar\ad\coreLie\rpar^t\lpar\coreLie\rpar$ is spanned by all powers of $A$ with exponent at least $t+1$. If $F$ is a basis element of $\coreLie$ from \eqref{coreLiebasis}, and if the integer $n$ is at least $k$, then $\lbrak F,A^n\rbrak=0\in\mathcal{G}_k$, if $F$ is a power of $A$. The other case is when $F=BA^2$. By routine application of \eqref{AnBm},
\begin{eqnarray}
\lbrak BA^2,A^n\rbrak = -nA^{n+1}.\label{gLieStruct}
\end{eqnarray}
But since $n\geq k$, we have $n+1\geq k+1$, and the right-hand side of \eqref{gLieStruct} is an element of $\mathcal{G}_k$. We have thus shown $\lbrak\coreLie,\mathcal{G}_{k-1}\rbrak \subseteq \mathcal{G}_k$. The other set inclusion follows from \eqref{gLieStruct} which shows us that every spanning set element of $\mathcal{G}_k$ is equal to $\frac{-1}{n}$ times the Lie bracket of an element of $\coreLie$ with an element of $\mathcal{G}_{k-1}$. Hence, $\lbrak\coreLie,\mathcal{G}_{k-1}\rbrak=\mathcal{G}_k$, and the induction is complete. As a consequence, every Lie subalgebra $\lpar\ad\coreLie\rpar^k\lpar\coreLie\rpar$ in the lower central series for $\coreLie$ is not the zero Lie algebra, so $\coreLie$ is not nilpotent. However, the derived (Lie) algebra $\lpar\ad\coreLie\rpar\lpar\coreLie\rpar$ is spanned by powers of $A$, and is hence abelian. Thus, the next Lie subalgebra in the derived series is already the zero Lie algebra. Therefore, $\coreLie$ is solvable. This completes the proof. 
\end{proof}

\begin{corollary}  The generators of $\coreLie$, together with their images under $\isoH$, generate $\HWeyl$ as a Lie algebra.
\end{corollary}
\begin{proof} Let $\HWeyl'$ be the Lie subalgebra of $\HWeyl$ generated by $A$, $BA^2$, $\isoH\lpar A\rpar=B$ and $\isoH\lpar BA^2\rpar=AB^2$, where the last generator, by routine application of \eqref{AnBm}, is equal to $B^2A+2B$. Thus, $B^2A=\isoH\lpar BA^2\rpar-2\isoH(A)\in \HWeyl'$. Since all generators of $\HWeyl$ from Theorem~\ref{GenThm} are in $\HWeyl'$, we have $\HWeyl\subseteq\HWeyl'$, but since $\HWeyl'$ is a Lie subalgebra of $\HWeyl$, $\HWeyl=\HWeyl'$.
\end{proof}

\begin{theorem} $\HWeyl\  =\  \lpar\coreLie\oplus\isoH(\coreLie)\rpar\  +\  \lbrak\coreLie,\isoH(\coreLie)\rbrak$.
\end{theorem}
\begin{proof} First, we show that every basis element $B^mA^n$ of $\HWeyl$ from \eqref{thebasis} is an element of $\coreLie+\isoH(\coreLie)+\lbrak\coreLie,\isoH(\coreLie)\rbrak$, and to do this, we use induction on $m+n$. Suppose that for any nonnegative integer $t<m+n$, any $B^iA^j$ with $i+j=t$ is an element of\linebreak $\coreLie+\isoH(\coreLie)+\lbrak\coreLie,\isoH(\coreLie)\rbrak$. By routine computations that make use of \eqref{AnBm}, \eqref{HWisomap},
{\scriptsize\begin{eqnarray}
B^mA^n & = &  \frac{-1}{(m+1)(n+1)}\lbrak A^{n+1},\isoH\lpar A^{m+1}\rpar\rbrak\label{BmAnSum1}\\
& & +\frac{1}{(m+1)(n+1)}\sum_{k=2}^{\min\{m+1,n+1\}}{{m+1}\choose k}{{n+1}\choose k}k!\  B^{m+1-k}A^{n+1-k},\label{BmAnSum2}
\end{eqnarray}}

\noindent where, in \eqref{BmAnSum2}, the sum of the exponents of $B$ and $A$ range from either $m-n$ or $n-m$, up to $m+n-2$. All such sums of exponents are strictly less than $m+n$. By the inductive hypothesis, every $B^{m+1-k}A^{n+1-k}$ in \eqref{BmAnSum2} is an element of the sum\linebreak $\coreLie+\isoH(\coreLie)+\lbrak\coreLie,\isoH(\coreLie)\rbrak$, and so is any linear combination of them, such as the summation in \eqref{BmAnSum2}. Also, the Lie bracket in \eqref{BmAnSum1} is an element of $\lbrak\coreLie,\isoH(\coreLie)\rbrak$. Thus, $B^mA^n$ is an element of $\coreLie+\isoH(\coreLie)+\lbrak\coreLie,\isoH(\coreLie)\rbrak$, and by induction, so is any basis element of $\HWeyl$ from \eqref{thebasis}. Hence, $\HWeyl$ is contained in the sum $\coreLie+\isoH(\coreLie)+\lbrak\coreLie,\isoH(\coreLie)\rbrak$, every summand in which, is a vector subspace of $\HWeyl$. Therefore, $\HWeyl=\coreLie+\isoH(\coreLie)+\lbrak\coreLie,\isoH(\coreLie)\rbrak$.

What remains to be shown is that the sum of $\coreLie$ and $\isoH\lpar\coreLie\rpar$ is direct. The elements
\begin{eqnarray}
BA^2,\quad B^2A,\quad A^n,\quad B^n,\qquad (n\in\N\setdiff\{0\}),\label{coreLieCbasis0}
\end{eqnarray}
of $\HWeyl$ are among the basis elements \eqref{thebasis}. Thus, the elements \eqref{coreLieCbasis0} are linearly independent, and if we partition \eqref{coreLieCbasis0} into two: the basis elements of $\coreLie$ from Lemma~\ref{coreLieLem}, and
\begin{eqnarray}
B^2A,\quad B^n,\qquad (n\in\N\setdiff\{0\}),\label{coreLieCbasis}
\end{eqnarray}
then the linear span of \eqref{coreLieCbasis0} is equal to the direct sum $\coreLie\oplus\coreLieC$ where $\coreLieC$ is the linear span of \eqref{coreLieCbasis}. To complete the proof, we only need to show $\coreLieC=\isoH\lpar\coreLie\rpar$. The linear inpendence of \eqref{coreLieCbasis0} implies the linear independence of \eqref{coreLieCbasis}, and so the spanning elements \eqref{coreLieCbasis} of $\coreLieC$ form a basis for $\coreLieC$. Using \eqref{AnBm}, 
\begin{eqnarray}
A^2B & = & BA^2 + 2A,\nonumber\\
AB^2 & = & B^2A + 2B,\nonumber
\end{eqnarray}
and so, by \eqref{HWisomap},
\begin{eqnarray}
B^2A & = &\isoH(A^2B)=\isoH(BA^2+2A),\qquad \mbox{where }BA^2+2A\in\coreLie,\label{PHIg1}\\
B^n & = & \isoH(A^n),\label{PHIg2}\\
\isoH(BA^2) & = & - AB^2 = -B^2A - 2B\  \in\  \coreLieC,\label{PHIg3}
\end{eqnarray}
By \eqref{PHIg1}, \eqref{PHIg2}, every basis element of $\coreLieC$ is in $\isoH(\coreLie)$, while by \eqref{PHIg2}, \eqref{PHIg3}, the image, under $\isoH$, of every basis element of $\coreLie$ is in $\coreLieC$. Thus, $\coreLieC=\isoH(\coreLie)$, and this completes the proof.
\end{proof}

\subsection{Further directions} At this point,  one continuation of this study we can suggest is the exploration of the effect of intersection compositions \cite[Definition 4.1(i)]{bok07}, or alternatively, \cite[p. 38]{ufn95}, on the choice of the Lie subalgebra of $\HWeyl$ (perhaps different from $\coreLie$) which may be used to decompose $\HWeyl$ in terms of such a Lie subalgebra and of its image under $\isoH$. Another possibility is the extension, or the development of analogs, of the methods in this study for an arbitrary $q$-deformed Heisenberg algebra.



\begin{funding}
This work was supported by the Research and Grants Management Office, formerly the University Research Coordination Office (URCO), of De La Salle University, with grant number 15FU1TAY20-1TAY21.
\end{funding}


\end{document}